\documentclass[a4paper, 12pt,reqno]{amsart}

\usepackage[margin=1in]{geometry, stmaryrd}
\usepackage{bbm}
\usepackage{amssymb,latexsym}
\usepackage{graphicx}
\usepackage{tikz}
\usepackage{mathtools}
\usepackage{color}
\usepackage[hyphens]{url}
\usepackage[T1]{fontenc}
\usepackage{amsmath}
\usepackage{mathdots}
\usepackage[toc,page]{appendix}
\usepackage{url}    
\usepackage{hyperref}
\usepackage{breakurl}
\usepackage{comment}
\usepackage{mathrsfs}
\usepackage{fancyhdr} 
\usepackage{leftidx}
\allowdisplaybreaks
\DeclareMathOperator{\GL}{GL}

\DeclareMathOperator{\Hom}{Hom}
\DeclareMathOperator{\N}{N}
\DeclareMathOperator{\A}{A}

\DeclareMathOperator{\E}{E}
\DeclareMathOperator{\F}{F}

\DeclareMathOperator{\K}{K}
\DeclareMathOperator{\U}{U}

\DeclareMathOperator{\J}{J}

\DeclareMathOperator{\W}{W}
\DeclareMathOperator{\C}{C}
\DeclareMathOperator{\f}{f}
\DeclareMathOperator{\e}{e}
\renewcommand\L{\operatorname{L}}
\DeclareMathOperator{\tr}{tr}
\DeclareMathOperator{\Mat}{Mat}
\DeclareMathOperator{\Ind}{Ind}
\DeclareMathOperator{\ind}{ind}

\DeclareMathOperator{\vol}{vol}
\DeclareMathOperator{\val}{val}
\DeclareMathOperator{\Gal}{Gal}
\DeclareMathOperator{\disc}{disc}
\DeclareMathOperator{\Res}{Res}

\newcommand{\sfrac}{\genfrac{}{}{}1}

\newcommand{\bigslant}[2]{%
  \raisebox{.2em}{$#1$}\left/\raisebox{-.2em}{$#2$}\right.%
}

\newcommand{\bigbackslant}[2]{%
  \raisebox{-.2em}{$#1$}\left\backslash\raisebox{.2em}{$#2$}\right.%
}

\newcommand{\smallbackslant}[2]{%
  \raisebox{-.08em}{$\scriptstyle #1$}\,
  \backslash\,
  \raisebox{.08em}{$\scriptstyle #2$}%
}

\newcommand{\bigdoublecoset}[3]{%
  \raisebox{-.2em}{$#1$}
  \left\backslash
  \raisebox{.2em}{$#2$}
  \middle/
  \raisebox{-.2em}{$#3$}
  \right.%
}

\newcommand{\smalldoublecoset}[3]{%
  \raisebox{-.08em}{$\scriptstyle #1$}\!
  \backslash\!
  \raisebox{.08em}{$\scriptstyle #2$}\!
  /\,%
  \raisebox{-.08em}{$\scriptstyle #3$}%
}

\theoremstyle{plain}
\numberwithin{equation}{section}
\newtheorem{theorem}[equation]{Theorem}
\newtheorem{lemma}[equation]{Lemma}
\newtheorem{corollary}[equation]{Corollary}

\newtheorem{definition}[equation]{Definition}
\newtheorem{proposition}[equation]{Proposition}

\newtheorem{remark}[equation]{Remark}

\title{T\lowercase{he} L\lowercase{ocal} L\lowercase{anglands} C\lowercase{orrespondence} \lowercase{for} M\lowercase{iddle} S\lowercase{upercuspidal} R\lowercase{epresentations} \lowercase{of} $p$\lowercase{-adic} $\GL(2n)$}

\author{D\lowercase{avid} C. L\lowercase{uo} \lowercase{and}  S\lowercase{haun} S\lowercase{tevens}}
\date{}

\address{School of Mathematics, University of Minnesota, Minneapolis, MN 55455, United States} \email{luo00275@umn.edu}

\address{School of Engineering, Mathematics and Physics, University of East Anglia, Norwich Research Park, Norwich NR4 7TJ, United Kingdom} \email{Shaun.Stevens@uea.ac.uk}



\begin{document}

\begin{abstract}
    Let $\F$ be a non-archimedean local field of characteristic zero with residual characteristic $p$.~In this paper we give an explicit description of the local Langlands correspondence for middle supercuspidal representations of $\GL(2n,\F)$, under the tameness condition $p\nmid 2n$, in terms of the maximal simple types that define them.~We achieve this by explicitly computing and comparing the gamma factors on the automorphic and Galois sides of the local Langlands correspondence.~The computation on the automorphic side requires neither the tameness condition nor the characteristic zero condition.
\end{abstract}

\maketitle

\section{Introduction}

Let $\F$ be a non-archimedean local field of characteristic zero with residual characteristic $p$, and let $\GL(n, \F)$ denote the general linear group of degree $n$ with entries in $\F$.~The local Langlands correspondence for $\GL(n, \F)$, established by Harris--Taylor and Henniart, is a fundamental result relating the representation theory of $\GL(n, \F)$ to Galois-theoretic data \cite{Harris, Henniart1}.~More precisely, it gives a bijection between isomorphism classes of smooth irreducible representations of $\GL(n, \F)$ and equivalence classes of $n$-dimensional Weil--Deligne representations (called \textit{local Langlands parameters}) of $\F$.~Under this correspondence, isomorphism classes of supercuspidal representations correspond precisely to equivalence classes of irreducible $n$-dimensional representations of the Weil group $\W_{\F}$ of $\F$.~Although the existence of this correspondence is known in full generality, making the bijection explicit remains a subtle problem.

Through a series of papers, Bushnell and Henniart answered this question for essentially tame supercuspidal representations of $\GL(n, \F)$, in terms of admissible pairs and the maximal simple types defined by them \cite{BH-4, BH-5, BH-6}.~In particular, the essentially tame setting includes the case when $p \nmid n$.~Beyond this situation, Bushnell and Henniart also gave an explicit description of the Langlands parameters for \textit{epipelagic}, or \textit{simple}, \textit{supercuspidal representations} of $\GL(n, \F)$, with no restriction on $p$ (see also~\cite{Adrian-Liu} for the case when $p \nmid n$) \cite{BH-1}.~Furthermore, they provide a mechanism to deduce an explicit correspondence in general from the totally wild case (when $n$ is a power of $p$) \cite{BH-8}.

Recently, the authors introduced a new class of depth $\frac{1}{n}$ supercuspidal representations of $\GL(2n, \F)$, called \textit{middle supercuspidal representations} \cite{Luo}.~These representations are minimax in the sense of \cite{Adrian}, and they have the next smallest positive depth after the simple supercuspidal representations.~Alternatively, in terms of the Galois side of the local Langlands correspondence, they are the depth $\frac{1}{n}$ supercuspidal representations whose Langlands parameter has multiplicity-free restriction to wild inertia and twisting number two (i.e. there are precisely two unramified quasi-characters of~$\F^{\times}$ which leave the parameter invariant by twisting).

The goal of this article is to explicitly describe the local Langlands correspondence for essentially tame middle supercuspidal representations of $\GL(2n, \F)$ using computations of twisted gamma factors.~More precisely, the main contribution is to give a closed formula for certain \textit{twisted gamma factors} -- recall that middle supercuspidal representations are determined up to isomorphism by their \textit{twisted gamma factors} with tamely ramified quasi-characters of $\F^\times$ and simple supercuspidal representations of $\GL(n,\F)$ \cite[Theorem 1.1]{Luo}.~As a result, in the tame setting we give a closed formula for the admissible quasi-character whose induction gives the corresponding Langlands parameter, in terms of the type-theoretic data of the middle supercuspidal representation.

On the automorphic side, we assume no characteristic condition on $\F$, nor any tameness conditions on $p$.~We compute the relevant twisted gamma factors using carefully chosen Whittaker functions in the Whittaker models of simple and middle supercuspidal representations.~Moving to the Galois side, we assume the tameness condition $p\nmid 2n$.~We then compute the corresponding gamma factors attached to Langlands parameters by realizing these counterparts as irreducible representations induced from a character, and decomposing their tensor products over suitable double cosets.~Comparing the two sets of formulas then determines the quasi-character appearing in the induced irreducible Weil representation associated with a middle supercuspidal representation.

To explain our result precisely, we need to introduce some notation.~Let $\mathcal{O}_{\F}$ denote the valuation ring of $\F$ and let $k_{\F}$ be its residue field with characteristic $p$.~Furthermore, we denote the group of roots of unity of order prime to $p$ by $\mu_{\F}'$.~Lastly, we fix $\psi_{\F}$ to be an additive quasi-character of level one and $\varpi_{\F}$ a uniformizer of $\F$.

We parametrize isomorphism classes of middle supercuspidal representations by triples $\left(\bar f, \, \chi, \, \zeta \right)$, where $\bar{f}$ is a monic irreducible polynomial of degree two over $k_{\F}$, $\chi$ is a multiplicative quasi-character of the degree two extension 
\[
k_{\bar{f}} = \bigslant{k_{\F}[X]}{\left\langle \, \bar{f} \, \right\rangle}
\]
of $k_{\F}$, and $\zeta \in \mathbb{C}^{\times}$.~From $f \in \mathcal{O}_{\F}[X]$ a lift of $\bar{f}$, we obtain a middle supercuspidal representation $\pi_{(f, \, \chi, \, \zeta)}$ whose isomorphism class depends only on $(\bar{f}, \, \chi, \, \zeta)$.~We also associate to $f$ a tower of extensions
\[
\F \subset \text{L}_{f} \subset \E_{f} 
\]
where $\text{L}_{f}/\F$ is unramified of degree two,
$\E_f/\text{L}_f$ is totally ramified of degree $n$, together with a particular uniformizer $\varpi_{\E_{f}}$ (such that~$\varpi_{E_f}^{-n}\varpi_F$ has minimal polynomial~$f$).

Let $\lambda_{\E_{f}/\F}(\psi_{\F})$ denote the Langlands constant and $\varkappa_{\E_{f}/\text{L}_{f}}$ the quadratic quasi-character of $\W_{\text{L}_{f}}$ associated to the extension~$\E_f/\text{L}_f$ (see~\eqref{eqn:varkappa}).~Decomposing $\E_{f}^{\times}$ as
\[
    \E_f^{\times}
    =
    \langle\varpi_{\E_{f}}\rangle
    \times \mu_{\normalfont{\text{L}_{f}}}'
    \times\left(1+\mathcal{P}_{\E_f}\right),
\]
we state our main result. 

\begin{theorem}\label{TheoremMain}
Let $\pi_{(f, \, \chi, \, \zeta)}$ be an essentially tame middle supercuspidal representation of $\GL(2n, \F)$, and let $\Phi_{(f, \, \chi, \, \zeta)}$ denote its Langlands parameter.~Then
\[
   \Phi_{(f, \, \chi, \, \zeta)} = \Ind_{\W_{\E_{f}}}^{\W_{\F}} \, \xi_{(f, \, \chi, \, \zeta)}
\]
where $\xi := \xi_{(f, \, \chi, \, \zeta)}$ is the quasi-character of $\E_{f}^\times$ determined as
follows.

\vspace{0.1in}

\begin{enumerate}
\item First,
\[
   \xi(\varpi_{\E_{f}}) = \zeta^{-1} \cdot \lambda_{\E_{f}/\F}(\psi_{\F}) \,  .
\]
\item Second, 
\[
   \xi \, \big|_{\, \mu_{\normalfont{\text{L}_{f}}}'}
   =
   \chi \otimes \left(\varkappa_{\E_{f}/{\normalfont \text{L}_{f}}} \, \big|_{\, \mu_{\normalfont{\text{L}_{f}}}'} \right) . 
\]
\item Finally, for $x \in \mathcal{P}_{\E_{f}}$, 
\[
   \xi(1+x) 
   =
   \psi_{\F} \circ \tr_{\E_{f}/\F}\left(\varpi_{\E_{f}}^{-1} \, x\right) .
\]
\end{enumerate}
\end{theorem}

\begin{remark}\normalfont
    The main point of our paper is less the result in Theorem~\ref{TheoremMain} but rather the method we have used to prove it -- in particular, the calculation of twisted gamma factors on the automorphic side \emph{without any tameness hypotheses}, which we hope and expect to be able to apply in more general situations.
    
    Indeed, the explicit description of the character in Theorem~\ref{TheoremMain} can also be recovered by tracing through the work of Bushnell--Henniart and Adrian--Liu.~On the one hand, the middle supercuspidal representation~$\pi_{(f, \, \chi, \, \zeta)}$ is obtained by automorphic induction from a simple supercuspidal representation of~$\GL(n, \L_f)$.~On the other hand, in the tame case, all these representations are given by admissible pairs (following a recipe described explicitly in~\cite{BH-4}), while Bushnell--Henniart also explain how tame automorphic induction affects admissible pairs.~Putting this together with the results of~\cite{Adrian-Liu}, we obtain the same result (after some easy simplifications).
    
    What our results should allow is to make a comparison with the Galois parameter even if~$p\mid 2n$, as in the case of simple supercuspidals in~\cite{BH-1}.
\end{remark}

\begin{remark}\normalfont
    Our initial computation of the Langlands parameter gave a rather ugly and complicated formula which, though explicit, was somewhat unsatisfactory.~Prompted by the fact that we could use the method described in the previous remark to arrive at a simpler formula, we asked ChatGPT if it could help in  simplifying our formula -- and the answer was yes.~We have written our own account of this simplification in the appendix (including writing more general statements where appropriate, and attributions to others' work), choosing to separate it from the main body of the paper to make clear where AI has assisted us.
\end{remark}

The paper is organized as follows.~In Section \ref{Notation}, we fix notation used throughout this article.~In Section \ref{AutomorphicSide}, we recall the definition of twisted gamma factors via Rankin--Selberg integrals, review the type-theoretic construction of simple and middle supercuspidal representations, give explicit Whittaker functions in the Whittaker models of these representations, and compute the twisted gamma factors needed to prove Theorem~\ref{TheoremMain}.~In Section~\ref{GaloisSide}, we turn to the Galois side.~After recalling admissible pairs, local constants, and the behavior of gamma factors under induction, we compute the gamma factors attached to the Langlands parameters corresponding to middle and simple supercuspidal representations.~Finally, in Subsection~\ref{LanglandsParameter}, we compare the automorphic and Galois gamma factors and deduce the formula for the Langlands parameter of an essentially tame middle supercuspidal representation. The account of the simplifications of the formula which AI helped us to find is in the appendix.

\section*{Acknowledgments}

The authors thank Dihua Jiang for helpful discussions.~The second-named author was supported by EPSRC grant \mbox{EP/V061739/1}.~As described above, the final simplification was assisted by ChatGPT.

\section{Notation}\label{Notation}

In this section, we introduce notation that we will use throughout this article.~Let $\F$ be a non-archimedean local field with $\mathcal{O}_{\F}$ its valuation ring, $k_{\F}$ its residue field with cardinality $q_{\F}$ and characteristic $p$, and $\mu_{\F}'$ the group of roots of unity of order prime to $p$.~Next, let $\mathcal{P}_{\F}$ denote the maximal ideal of $\mathcal{O}_{\F}$ and $\varpi_{\F}$ a fixed uniformizer.~Furthermore, let $\psi_{\F}$ be a non-trivial additive quasi-character of $\F$ with conductor $\mathcal{P}_{\F}$, i.e.~$\psi_{\F}$ is non-trivial on $\mathcal{O}_{\F}$ but trivial on $\mathcal{P}_{\F}$.~The parametrization of simple and middle supercuspidal representations in Section \ref{AutomorphicSide} will depend on our fixed choices of $\psi_{\F}$ and $\varpi_{\F}$.

Let $\E/\F$ be a finite extension.~We let $\e\left(\E \, | \, \F\right)$ and $\f\left(\E \, | \, \F \right)$ denote the ramification index and residue degree of $\E/\F$ respectively.~For an element $x \in \mathcal{O}_{\F}^{\times}$, we denote its reduction in $k_{\F}^{\times}$ by $\overline{x}$.

Let $\N(n, \F)$ denote the unipotent radical of the standard (upper triangular) Borel subgroup of $\GL(n, \F)$, let $\A(n, \F)$ denote the subgroup of diagonal matrices, and let $\text{P}(n, \F)$ denote the standard mirabolic subgroup.~Let $\Mat(n\times m,\F)$ denote the space of $n\times m$ matrices with entries in $\F$, and let $I_{n}$ denote the $n\times n$ identity matrix.~Furthermore, let $\psi_n$ denote the standard smooth non-degenerate quasi-character of $\N(n,\F)$ defined by
\[
\psi_n(u)=\psi_{\F}\left(\sum_{i=1}^{n-1}u_{i, \, i+1}\right)
\]
for $u=(u_{i,j})\in \N(n,\F)$.

We will fix throughout the article a self-dual Haar measure on $\F$, relative to $\psi_{\F}$.~In particular, this means $\int_{\mathcal{O}_{\F}}  dx = q_{\F}^{1/2}$.~Lastly, let $\Ind$ denote smooth induction, with compact induction being denoted by $\ind$. 

\section{Automorphic Side}\label{AutomorphicSide}

In this section, we first recall the definition of twisted gamma factors via Rankin--Selberg integrals, together with the constructions of simple and middle supercuspidal representations via type theory.~We then give explicit Whittaker functions in the Whittaker models of these representations, which are used in the subsequent computations.~Next, we recall the formula for the gamma factor of a middle supercuspidal representation of $\GL(2n, \F)$ twisted by a tamely ramified quasi-character of $\F^\times$.~Finally, we compute the twisted gamma factor for a pair consisting of a middle supercuspidal representation of $\GL(2n, \F)$ and a simple supercuspidal representation of $\GL(n, \F)$.~Throughout this section, we impose no restriction on $p$.

\subsection{Twisted gamma factors}\label{TwistedGammaFactors}

In this subsection, we define \textit{twisted gamma factors} via the Rankin--Selberg integrals.~First, we recall the following (\cite[Section 2]{Adrian-Liu} and \cite{BH-3}):~an irreducible admissible representation $\left(\pi, V_{\pi}\right)$ of $\GL(n, \F)$ is called \textit{generic} if 
\[
\Hom_{\GL(n, \, \F)}\left(\pi,  \Ind_{\N(n, \, \F)}^{\GL(n, \, \F)}\psi_{n}\right) \neq 0.
\]
By the uniqueness of local Whittaker models, this $\Hom$-space is at most 1-dimensional.~From Frobenius reciprocity,
\[
\Hom_{\GL(n, \, \F)}\left(\pi,  \Ind_{\N(n, \, \F)}^{\GL(n, \, \F)}\psi_{n}\right) \cong \Hom_{\N(n, \, \F)}\left(\pi\big|_{\N(n, \, \F)},  \psi_{n}\right).
\]
Therefore, $\Hom_{\N(n, \, \F)}\left(\pi \big|_{\N(n, \, \F)},  \psi_{n}\right)$ is also at most 1-dimensional.

Assume that $\left(\pi, V_{\pi}\right)$ is generic. We fix a non-zero functional 
\[
l \in \Hom_{\N(n, \, \F)}
\left(\left.\pi\right|_{\N(n, \, \F)}, \, \psi_{n}\right)
\]
which is unique up to scalar.~The \textit{Whittaker function} attached to a vector $v \in V_{\pi}$ is defined by
\[
W_{v}(g) := l \left(\pi(g)v\right), \, \text{ for all } g \in \GL(n, \F),
\]
so that $W_{v} \in \Ind_{\N(n, \, \F)}^{\GL(n, \, \F)}\psi_{n}$.~The space
\[
W(\pi, \psi_{n}) := \{ W_{v} : v \in V_{\pi}\}
\]
is called the \textit{Whittaker model} of $\pi$ and $\GL(n, \F)$ acts on it by right translation.~It is easy to see that the Whittaker model of $\pi$ is independent of the choice of the non-zero functional $l$.

Suppose that $\pi_{1}$ is a generic representation of $\GL(n, \F)$ and that $\pi_{2}$ is a generic representation of $\GL(m, \F)$, where $m<n$; let $\omega_{\pi_{1}}$ and $\omega_{\pi_{2}}$ denote their central characters, respectively.~Furthermore, let 
\[
w_{n, \, m} = \begin{pmatrix}
    I_{m} & \\
    & w_{n-m}
\end{pmatrix} \in \GL(n, \F), \text{ where \; } w_{r} = \begin{pmatrix}
    & & 1 \\
    & \iddots & \\
    1 & &
\end{pmatrix} \in \GL(r, \F).
\]

Next, let $W_{\pi_{1}} \in W\left(\pi_{1}, \psi_{n}\right)$ and $W_{\pi_{2}} \in W\left(\pi_{2}, \psi_{m}^{-1}\right)$.~We define the \textit{Rankin--Selberg integrals} $\widetilde{\Psi}$ and $\Psi$ attached to $\pi_{1}$ and $\pi_{2}$ by 
\begin{align*}
    & \widetilde{\Psi}\left(s; \, W_{\pi_{1}},  \, W_{\pi_{2}}\right) \\
    &= \int\displaylimits_{\smallbackslant{\N(m, \, \F)}{\GL(m, \, \F)}}\int\displaylimits_{\Mat(n-m-1 \times m, \, \F)}W_{\pi_{1}}\left(\begin{pmatrix}
        h & & \\
        x & I_{n-m-1} & \\
        & & 1
    \end{pmatrix}\right)W_{\pi_{2}}(h)|\det(h)|^{s-\frac{n-m}{2}} \, dx \, dh,
\end{align*}
and 
\[
\Psi\left(s; \, W_{\pi_{1}}, \, W_{\pi_{2}}\right) = \int\displaylimits_{\smallbackslant{\N(m, \, \F)}{\GL(m, \, \F)}} W_{\pi_{1}}\left(\begin{pmatrix}
    h & \\
    & I_{n-m}
\end{pmatrix}\right)W_{\pi_{2}}(h)|\det(h)|^{s-\frac{n-m}{2}} \, dh.
\]
These integrals are absolutely convergent for $\text{Re}(s)$ sufficiently large and are rational functions of $q_{\F}^{-s}$ \cite[Section 2.7]{JPSS}.

From these integrals, we obtain the following functional equation which defines the twisted gamma factor $\gamma(s, \, \pi_{1} \times \pi_{2}, \, \psi_{\F})$.
\begin{theorem}\cite[Section 2.7]{JPSS}\label{GammaFactorEquation}
    There is a rational function $\gamma(s, \, \pi_{1} \times \pi_{2}, \, \psi_{\F}) \in \mathbb{C}\left(q_{\F}^{-s}\right)$ such that
    \[
    \widetilde{\Psi}\left(1-s; \, \rho(w_{n, \, m})\widetilde{W}_{\pi_{1}}, \, \widetilde{W}_{\pi_{2}}\right) = \omega_{\pi_{2}}(-1)^{n-1} \, \gamma(s, \, \pi_{1} \times \pi_{2}, \, \psi_{\F})\Psi\left(s; \, W_{\pi_{1}},  \, W_{\pi_{2}}\right)
    \]
for all $W_{\pi_{1}} \in W\left(\pi_{1}, \psi_{n}\right)$, $W_{\pi_{2}} \in W\left(\pi_{2}, \psi_{m}^{-1}\right)$, where $\rho$ denotes the right translation action, and $\widetilde{W}_{\pi_{1}}(g) = W_{\pi_{1}}\left(w_{n} \, \leftidx^{t}g^{-1}\right)$, $\widetilde{W}_{\pi_{2}}(h) = W_{\pi_{2}}\left(w_{m} \, \leftidx^{t}h^{-1}\right)$, for $g \in \GL(n, \F)$ and $h \in \GL(m, \F)$.
\end{theorem}

\subsection{Simple supercuspidal representations}\label{SimpleSupercuspidalRepresentations}

In this subsection, we briefly recall the structure of \textit{simple supercuspidal representations} of $\GL(n, \F)$, $n \geq 2$, via the theory of maximal simple types (\cite{Knightly} and \cite[Example 2.18]{Ye}).~We then give an explicit Whittaker function associated with a simple supercuspidal representation.~Moreover, we carry out the construction here with respect to $\psi_{\F}^{-1}$ rather than $\psi_{\F}$, since our computations involving twisted gamma factors require us to consider Whittaker models of simple supercuspidals with respect to $\psi_{n}^{-1}$.

Let $\mathfrak{I}_{n}$ be the standard minimal hereditary $\mathcal{O}_{\F}$-order in $\mathrm{Mat}(n \times n, \F)$; that is, the space of $n \times n$ matrices with entries in $\mathcal{O}_{\F}$ whose image modulo $\mathcal{P}_{\F}$ is an upper triangular matrix with entries in $k_{\F}$.~Furthermore, let $\mathfrak{I}_{n}^{+}$ denote the Jacobson radical of $\mathfrak{I}_{n}$; that is, the space of $n \times n$ matrices with entries in $\mathcal{O}_{\F}$ whose image modulo $\mathcal{P}_{\F}$ is a strictly upper triangular matrix with entries in $k_{\F}$.

Next, let $v \in \mathcal{O}_{\F}^{\times}$ and 
\[
\beta_{v} = \begin{pmatrix}
    & & & & \frac{1}{v\varpi_{\F}} \\
    1 & & & & \\
     & 1 & & & \\
      & & \ddots & & \\
      & & & 1 &
\end{pmatrix}.
\]
Then $\E_{v} = \F \left[\beta_{v} \right]$ is a degree $n$ totally ramified extension of $\F$.~Furthermore, we have that $\text{v}_{\E_{v}}\left(\beta_{v}\right) = -1$ where $\text{v}_{\E_{v}}$ is the normalized discrete valuation on $\E_{v}$.~Let
\begin{equation}\label{MinimalPolynomialSimple}
m_{\beta_{v}}(X) = X^{n} - \frac{1}{v\varpi_{\F}}
\end{equation}
denote the minimal polynomial of $\beta_{v}$.

Associated with the above are the open compact subgroups of $\mathfrak{I}_{n}^{\times}$:
\[
\U^{1}\left(\mathfrak{I}_{n}\right) := I_{n} + \mathfrak{I}_{n}^{+}
\]
and 
\[
\J_{v} := \J\left(\beta_{v}, \mathfrak{I}_{n}\right) = \mathcal{O}_{\F}^{\times}  \U^{1}\left( \mathfrak{I}_{n}\right).
\]
Let $\psi_{\beta_{v}}$ be the quasi-character of $\U^{1}\left(\mathfrak{I}_{n}\right)$ defined by
\[
\psi_{\beta_{v}}(x) = \left(\psi_{\F}^{-1} \circ \tr_{\F}\right)\left(\beta_{v}(x-1)\right)
\]
where $\tr_{\F}$ denotes matrix trace.~Furthermore, let $\varphi$ be a quasi-character of $k_{\F}^{\times}$ inflated to $\mathcal{O}_{\F}^{\times}$.~From these quasi-characters, we may form a $\beta_{v}$-extension 
\[
\kappa_{(v,  \, \varphi)} \, \colon \, \J_{v} \longrightarrow \mathbb{C}^{\times}
\]
of $\psi_{\beta_{v}}$ by letting
\[
\kappa_{(v,  \, \varphi)}(xy) := \varphi(x) \, \psi_{\beta_{v}}(y)
\]
for $x \in \mathcal{O}_{\F}^{\times}$ and $y \in \U^{1}\left(\mathfrak{I}_{n}\right)$.

To form a maximal simple type with the data above, we take the pair $\left(\J_{v},  \kappa_{(v,  \, \varphi)}\right)$.~Moreover, to form an \textit{extended maximal simple type}, let $\textbf{J}_{v} = \E_{v}^{\times}\J_{v}$.~We extend $\kappa_{(v,  \, \varphi)}$ to a quasi-character $\Lambda_{(v,  \, \varphi,  \, \zeta)}$ on $\textbf{J}_{v}$ by setting $\Lambda_{(v,  \, \varphi,  \, \zeta)}\left(\beta_{v}\right) = \zeta \in \mathbb{C}^{\times}$ so that 
\[
\Lambda_{(v,  \, \varphi,  \, \zeta)}(\varpi_{\F}) = \left(\zeta^{n}  \, \varphi(v)\right)^{-1}.
\]
We call the pair $\left(\textbf{J}_{v}, \Lambda_{(v,  \, \varphi,  \, \zeta)}\right)$ an extended maximal simple type.~Hence, we may now define a simple supercuspidal representation.
\begin{definition}
    \normalfont 
    A \textit{simple supercuspidal representation} of $\GL(n, \F)$ is a supercuspidal representation of the form 
    \[
    \pi_{(v,  \, \varphi,  \, \zeta)} = \ind_{\textbf{J}_{v}}^{\GL(n, \, \F)}\Lambda_{(v,  \, \varphi, \, \zeta)}
    \]
    where $\left(\textbf{J}_{v}, \Lambda_{(v,  \, \varphi,  \, \zeta)}\right)$ is an extended maximal simple type as above.
\end{definition}
Simple supercuspidal representations are those supercuspidals of $\GL(n, \F)$ with minimal positive depth $\frac{1}{n}$.~Let $\omega_{(v,  \, \varphi,  \, \zeta)}$ denote the central character of $\pi_{(v,  \, \varphi,  \, \zeta)}$.

We parametrize simple supercuspidal representations here via $(\bar{v},  \, \varphi,  \, \zeta)$ because there is a bijection between the set of isomorphism classes $\mathcal{A}_{\text{simple}}^{(n)}$ of simple supercuspidal representations of $\GL(n, \F)$ and the set of reductions of such triples \cite[Proposition 1.3]{Imai}.
\begin{proposition}\label{BijectionSimple}
    There exists a bijection between $\mathcal{A}_{\text{simple}}^{(n)}$ and the set of triples $(\bar{v},  \, \varphi,  \, \zeta)$ with $\bar{v} \in k_{\F}^{\times}$, $\varphi$ a quasi-character of $k_{\F}^{\times}$, and $\zeta \in \mathbb{C}^{\times}$.
\end{proposition}

Lastly, we define an explicit Whittaker function $\mathcal{W}_{(v,  \, \varphi,  \, \zeta)} \in W\left(\pi_{(v,  \, \varphi,  \, \zeta)}, \psi_{\F}^{-1}\right)$ via the following \cite[Section 3.3]{Adrian-Liu}, \cite[Example 2.23]{Ye}:

\hspace{0pt}\resizebox{1.0\linewidth}{!}{
  \begin{minipage}{\linewidth}
\begin{align*}
    g \mapsto  \begin{cases} 
      \hfill \psi_{n}^{-1}(u)  \Lambda_{\left(v, \, \varphi, \, \zeta\right)}(h) &, \text{ if } g = uh \in \N(n, \F) \, \textbf{J}_{v} \text{ with } u \in \N(n, \F), \, h \in \textbf{J}_{v} \\
      \hfill 0 & ,  \text{ otherwise}
   \end{cases}.
\end{align*}
\end{minipage}
}

\subsection{Middle supercuspidal representations}\label{MiddleSupercuspidalRepresentations}

In this subsection, we briefly recall the construction of \textit{middle supercuspidal representations} of $\GL(2n, \F)$, $n \geq 2$, via the theory of maximal simple types \cite[Subsection 3.1]{Luo}.~Next, we give an explicit Whittaker function associated with a middle supercuspidal representation. 

Let $\mathfrak{A}_{2n}$ be the hereditary $\mathcal{O}_{\F}$-order in $\Mat(2n \times 2n, \F)$:
\begin{equation}\label{MiddleOrder}
    \mathfrak{A}_{2n} = \begin{pmatrix}
        \mathfrak{I}_{n} & \varpi_{\F}^{-1}\mathfrak{I}_{n}  \\
        \varpi_{\F}\mathfrak{I}_{n} & \mathfrak{I}_{n}  
    \end{pmatrix}
\end{equation}
with Jacobson radical
\begin{equation}\label{MiddleJacobson}
    \mathfrak{P}_{2n} = \begin{pmatrix}
        \mathfrak{I}_{n}^{+} & \varpi_{\F}^{-1}\mathfrak{I}_{n}^{+}  \\
        \varpi_{\F}\mathfrak{I}_{n}^{+} & \mathfrak{I}_{n}^{+}  
    \end{pmatrix}.
\end{equation}
Let 
\[
f(X) = X^{2}-dX-c \in \mathcal{O}_{\F}[X]
\]
such that $f \pmod{\mathcal{P}_{\F}}$ is irreducible over $k_{\F}$, and
\begin{equation*}
    \beta_{f} := \begin{pmatrix}
    & & & & & \frac{c}{\varpi_{\F}^{2}} \\
    1 & & & & &\\
    & \ddots & & & &  \\
    & & 1 & & & \frac{d}{\varpi_{\F}}\\
    &  & & \ddots  & & \\
    & & & & 1 &
\end{pmatrix}
\end{equation*}
where $\frac{d}{\varpi_{\F}} = \left(\beta_{f}\right)_{n+1, \,   2n}$.~Then $\E_{f} = F\left[\beta_{f}\right]$ is a degree $2n$ extension of $\F$ with $\e\left(\E_{f} \, | \, \F\right) = n$ and $\f\left(\E_{f} \, | \, \F\right) = 2$.

Furthermore, we have that $\text{v}_{\E_{f}}\left(\beta_{f}\right) = -1$ where $\text{v}_{\E_{f}}$ is the normalized discrete valuation on $\E_{f}$.~Let $\sigma_{f} := \beta_{f}^{n}\varpi_{\F}$ be a zero of $f$ and $\text{L}_{f} := F\left[\sigma_{f}\right]$ the unique degree two unramified extension of $\F$.~Moreover, let
\begin{equation}\label{MinimalPolynomialMiddle}
m_{\beta_{f}}(X) = X^{2n} - \frac{d}{\varpi_{\F}}X^{n} - \frac{c}{\varpi_{\F}^{2}}
\end{equation}
denote the minimal polynomial of $\beta_{f}$.

Associated with the above are the open compact subgroups of $\mathfrak{A}_{2n}^{\times}$:
\[
\U^{1}\left(\mathfrak{A}_{2n}\right) := I_{2n} + \mathfrak{P}_{2n}
\]
and 
\[
\J_{f} = \mathcal{O}_{\text{L}_{f}}^{\times} \U^{1}\left(\mathfrak{A}_{2n}\right).
\] 
If we fix the ordered $\mathcal{O}_{\F}$-basis $\{c, \, \sigma_{f}\}$ of $\mathcal{O}_{\text{L}_{f}}$, a general element $a_{0}c+a_{1}\sigma_{f} \in \mathcal{O}_{\text{L}_{f}}^{\times}$ is of the form
\begin{equation}\label{decomposition}
\begin{pmatrix}
        a_{0}c & & & & \frac{a_{1}c}{\varpi_{\F}} & & &\\
         & a_{0}c & & & & \frac{a_{1}c}{\varpi_{\F}} & & \\
         & & \ddots & & & & \ddots & \\
         &  & & a_{0}c & & & & \frac{a_{1}c}{\varpi_{\F}} \\
        a_{1}\varpi_{\F} & & & & a_{0}c+a_{1}d & & & \\
          & a_{1}\varpi_{\F} & & &  & a_{0}c+a_{1}d & & \\
          & & \ddots & &  & & \ddots & \\
         & & &  a_{1}\varpi_{\F} & & & &a_{0}c+a_{1}d 
    \end{pmatrix}
\end{equation}
\noindent where $a_{0}$, $a_{1} \in \mathcal{O}_{\F}$ are not both in $\mathcal{P}_{\F}$. 

Let $\psi_{\beta_{f}}$ be the quasi-character of $\U^{1}\left(\mathfrak{A}_{2n}\right)$ defined by
\[
\psi_{\beta_{f}}(x) = \left(\psi_{\F} \circ \tr_{\F}\right)\left(\beta_{f}(x-1)\right),
\]
and $\chi$ be a quasi-character of $k_{\text{L}_{f}}^{\times}$ inflated to $\mathcal{O}_{\text{L}_{f}}^{\times}$.~Then we may define the $\beta_{f}$-extension 
\[
\kappa_{\left(f, \, \chi \right)} \, \colon \, \J_{f} \longrightarrow \mathbb{C}^{\times}
\]
of $\psi_{\beta_{f}}$ by letting
\[
\kappa_{\left(f, \, \chi \right)}\left(xy\right) = \chi(x) \, \psi_{\beta_{f}}(y)
\]
for $x \in \mathcal{O}_{\text{L}_{f}}^{\times}$ and $y \in \U^{1}\left(\mathfrak{A}_{2n}\right)$.

To form a maximal simple type with the data above, we take the pair $\left(\J_{f},  \kappa_{(f,  \, \chi)}\right)$.~Moreover, to form an \textit{extended maximal simple type}, let $\textbf{J}_{f} = \E_{f}^{\times}\J_{f}$.~We extend $\kappa_{(f,  \, \chi)}$ to a quasi-character $\Lambda_{(f,  \, \chi,  \, \zeta)}$ on $\textbf{J}_{f}$ by setting $\Lambda_{(f,  \, \chi,  \, \zeta)}\left(\beta_{f}\right) = \zeta \in \mathbb{C}^{\times}$ so that 
\[
\Lambda_{(f,  \, \chi,  \, \zeta)}(\varpi_{\F}) = \zeta^{-n}  \, \chi(\sigma_{f}).
\]
The pair $\left(\textbf{J}_{f}, \Lambda_{(f,  \, \chi,  \, \zeta)}\right)$ is an extended maximal simple type.~Hence, we may now define a middle supercuspidal representation.
\begin{definition}
    \normalfont 
    A \textit{middle supercuspidal representation} of $\GL(2n, \F)$ is a supercuspidal representation of the form 
    \[
    \pi_{(f,  \, \chi,  \, \zeta)} = \ind_{\textbf{J}_{f}}^{\GL(2n, \, \F)}\Lambda_{(f,  \, \chi,  \, \zeta)}
    \]
    where $\left(\textbf{J}_{f}, \Lambda_{(f,  \, \chi,  \, \zeta)}\right)$ is an extended maximal simple type as above.
\end{definition}
Middle supercuspidal representations are those minimax supercuspidals of $\GL(2n, \F)$ with positive depth $\frac{1}{n}$.~Let $\omega_{(f,  \, \chi,  \, \zeta)}$ denote the central character of $\pi_{(f,  \, \chi,  \, \zeta)}$.

We parametrize middle supercuspidal representations here via $(f,  \, \chi,  \, \zeta)$ because there is a bijection between the set of isomorphism classes $\mathcal{A}_{\text{middle}}^{(2n)}$ of middle supercuspidal representations of $\GL(2n, \F)$ and the set of reductions of such triples \cite[Proposition 3.2]{Luo}.~For the following proposition, let 
\[
k_{\bar{f}} = \bigslant{k_{\F}[X]}{\left\langle \,\bar{f} \, \right\rangle}.
\]

\begin{proposition}\label{BijectionMiddle}
    There exists a bijection between $\mathcal{A}_{\text{middle}}^{(2n)}$ and the set of triples $(\bar{f},  \, \chi,  \, \zeta)$ with $\bar{f}$ a monic irreducible degree two polynomial over $k_{\F}$, $\chi$ a quasi-character of $k_{\bar{f}}^{\times}$, and $\zeta \in \mathbb{C}^{\times}$.
\end{proposition}

Lastly, we define an explicit Whittaker function $\mathcal{W}_{(f,  \, \chi,  \, \zeta)} \in W\left(\pi_{(f,  \, \chi,  \, \zeta)}, \psi_{\F}\right)$ via the following \cite[Subsection 3.2]{Luo}:

\hspace{0pt}\resizebox{1.0\linewidth}{!}{
  \begin{minipage}{\linewidth}
\begin{align*}
    g \mapsto  \begin{cases} 
      \hfill \psi_{2n}(u)  \Lambda_{(f,  \, \chi,  \, \zeta)}(h) &, \text{ if } g = uh \in \N(2n, \F) \, \textbf{J}_{f} \text{ with } u \in \N(2n, \F), \, h \in \textbf{J}_{f} \\
      \hfill 0 & ,  \text{ otherwise}
   \end{cases}.
\end{align*}
\end{minipage}
}

\subsection{Twisting by tamely ramified quasi-characters}\label{TamelyRamifiedAutomorphic}

In this subsection, we recall the formula for the twisted gamma factor between a middle supercuspidal of $\GL(2n, \F)$ and a tamely ramified quasi-character $\nu$ of $\F^{\times}$, i.e.~$\nu$ is trivial on $1 + \mathcal{P}_{\F}$ \cite[Corollary 3.7]{Luo}.
\begin{proposition}\label{AutomorphicTamelyRamified}
Let $\pi_{(f, \, \chi, \, \zeta)}$ be a middle supercuspidal representation of $\GL(2n, \F)$ and $\nu$ a tamely ramified quasi-character of $\F^{\times}$.
\[
\gamma\left(s, \, \pi_{(f,  \, \chi,  \, \zeta)} \times \nu, \, \psi_{\F}\right) = \zeta^{-1} \, \nu(-c^{-1}\varpi_{\F}^{2}) \, q_{\F}^{1-2s}.
\]
\end{proposition}

\subsection{Twisting by simple supercuspidal representations}\label{SimpleSupercuspidalAutomorphic}
In this subsection, we compute the twisted gamma factor between a middle supercuspidal $\pi_{(f, \, \chi, \, \zeta)}$ of $\GL(2n, \F)$ and a simple supercuspidal $\pi_{(v, \, \varphi, \, \zeta')}$ of $\GL(n, \F)$.~We set $\mathcal{W} := \mathcal{W}_{(f,  \, \chi,  \, \zeta)}$ and $\mathcal{W}' := \mathcal{W}_{\left(v,  \, \varphi,  \, \zeta'\right)}$.~Let 
\[
g_{v} := \begin{pmatrix}
    1 & & \frac{1}{v\varpi_{\F}} & &  \\
    & \ddots & & &  \\
    & & 1 & &  \\
    & & & \ddots & \\
    & & & & 1
\end{pmatrix}
\]
where $\frac{1}{v\varpi_{\F}} = (g_{v})_{1, \, n+1}$ and let $\pi_{(f, \, \chi, \, \zeta)}(g_{v})   \mathcal{W} \in W\left(\pi, \psi_{\F}\right)$ be the Whittaker function given by right translation of $\mathcal{W}$ by $g_{v}$.~From Theorem \ref{GammaFactorEquation}, we have
\begin{align}\label{GammaSimple}
    & \gamma\left(s, \, \pi_{(f, \, \chi, \, \zeta)}\times \pi_{(v, \, \varphi, \, \zeta')}, \, \psi_{\F}\right) \\ \notag
    & = \omega_{(v,\,\varphi,\,\zeta')} (-1)^{2n-1}\left(\frac{\widetilde{\Psi}\left(1-s; \, \rho(w_{2n, n})\widetilde{\pi_{(f, \, \chi, \, \zeta)}(g_{v})  \mathcal{W}}, \, \widetilde{\mathcal{W}'}\right)}{\Psi\left(s; \, \pi_{(f, \, \chi, \, \zeta)}(g_{v})   \mathcal{W}, \, \mathcal{W}'\right)}\right) \\ \notag
    &= \omega_{(v,\,\varphi,\,\zeta')} (-1)\left(\frac{\widetilde{\Psi}\left(1-s; \, \rho(w_{2n, n})\widetilde{\pi_{(f, \, \chi, \, \zeta)}(g_{v})  \mathcal{W}}, \, \widetilde{\mathcal{W}'}\right)}{\Psi\left(s; \, \pi_{(f, \, \chi, \, \zeta)}(g_{v})   \mathcal{W}, \, \mathcal{W}'\right)}\right). 
\end{align}

Recall that $\Psi\left(s; \, \pi_{(f, \, \chi, \, \zeta)}(g_{v}) \mathcal{W}, \, \mathcal{W}'\right)$ is given by the following proposition \cite[Proposition 3.9]{Luo}.
\begin{proposition}\label{AutomorphicProposition1}
\normalfont
    $\Psi\left(s; \, \pi_{(f, \, \chi, \, \zeta)}(g_{v})   \mathcal{W}, \, \mathcal{W}'\right) = \vol\left(\bigbackslant{\N\left(n, \F\right)}{\N\left(n, \F\right) \U^{1}\left(\mathfrak{I}_{n}\right)}\right)$.
\end{proposition}
Next, we compute $\widetilde{\Psi}\left(1-s; \, \rho(w_{2n, n})\widetilde{\pi_{(f, \, \chi, \, \zeta)}(g_{v})\mathcal{W}}, \, \widetilde{\mathcal{W}'}\right)$.
\begin{proposition}\label{AutomorphicProposition2}
\begin{align*}
  & \widetilde{\Psi}\left(1-s; \, \rho(w_{2n, n})\widetilde{\pi_{(f, \, \chi, \, \zeta)}(g_{v})\mathcal{W}}, \, \widetilde{\mathcal{W}'}\right) \\
  &= \zeta^{-n} \, \chi\left(-\frac{cv + \sigma_{f}}{vf\left(v^{-1}\right)}\right) \, \omega_{(v,\,\varphi,\,\zeta')}\left(-\frac{\varpi_{\F}^{2}}{f\left(v^{-1}\right)}\right) \, q_{\F}^{n(1-2s)} \, \vol\left(\bigbackslant{\N\left(n, \F\right)}{\N\left(n, \F\right) \U^{1}\left(\mathfrak{I}_{n}\right)}\right).
\end{align*}
\end{proposition}
\begin{remark}
    \normalfont
    The beginning of the proof is to use a change-of-variables argument from~\cite[Proof of Theorem 3.1]{Ye2}.~We note that this contains a minor error:~the exponent of $|\det(h)|$ should be $s - \frac{n-m}{2}$ instead of $s - 2 + \frac{n-m}{2}$.
\end{remark}
\begin{proof}
Using a change of variables when $\text{Re} (s) \ll 0$ \cite[Proof of Theorem 3.1]{Ye2}, we have that 
\begin{align}\label{TildePsiEq1}
    & \widetilde{\Psi}\left(1-s; \, \rho(w_{2n, n})\widetilde{\pi_{(f, \, \chi, \, \zeta)}(g_{v})\mathcal{W}}, \, \widetilde{\mathcal{W}'}\right)  \\ \notag
    &= \int\displaylimits_{\smallbackslant{\N(n, \, \F)}{\GL(n, \, \F)}}\,\int\displaylimits_{\Mat(n \times n-1, \, \F)}  \mathcal{W}\left(\begin{pmatrix}
          & 1 & \\
          &   & I_{n-1} \\
          h & & x
    \end{pmatrix}g_{v}\right)\mathcal{W}'(h)|\det(h)|^{s-\frac{n}{2}} \, dx  \, dh.
\end{align}
Let 
\begin{equation}\label{GL(2)}
    \alpha(h, x) := \begin{pmatrix}
          & 1 & \\
          &   & I_{n-1} \\
          h & & x
    \end{pmatrix}g_{v}
\end{equation}
with $h = (h_{i, j}) \in \N(n, \F) \backslash \GL(n, \F)$ and $x \in \Mat(n \times n-1, \F)$.~From \cite[Proposition 4.5.2 (Iwasawa decomposition)]{Bump}, we may choose representatives $h \in \A(n, \F) \GL(n, \mathcal{O}_{\F})$.~If $\alpha(h, x)$ lies in the support of $\mathcal{W}$, then $\alpha(h, x)$ lies in the double coset
\[
\N(2n, \F) \, \beta_{f}^{-n} \, \J_{f}
\]
\cite[Lemma 3.11, p.~17]{Luo}.

Suppose that $\alpha(h, x)$ lies in the support of $\mathcal{W}$.~Then 
\[
\alpha(h, x) = u \, \beta_{f}^{-n} \, (a_{0}c + a_{1}\sigma_{f}) \, z(h, x)
\]
for some $u \in \N(2n, \F)$, $a_{0}c + a_{1}\sigma_{f} \in \mathcal{O}_{\text{L}_{f}}^{\times}$, and $z(h, x) \in \U^{1}\left(\mathfrak{A}_{2n}\right)$ with $a_{0} \equiv a_{1}v \pmod{\mathcal{P}_{\F}}$ and $a_{1} \equiv a(v) \pmod{\mathcal{P}_{\F}}$ where
\[
a(v) := \frac{v}{cv^{2}+dv-1}
\] 
\cite[Proof of Proposition 3.10]{Luo}.~Let $\C_{1}(h)$ be the $n \times 1$ column vector given by the first column of $h$, and 
\[
\C(v, h) := (v\varpi_{\F})^{-1}\C_{1}(h).
\]
Expanding $\alpha(h, x)$, we have that it is of the form \cite[Proof of Proposition 3.10]{Luo}:
\[
\alpha(h, x) = \begin{pmatrix}
          & 1 & \\
          &   & I_{n-1} \\
          h & \C(v, h) & x
    \end{pmatrix}
\]
with $h \in a(v)v\varpi_{\F}^{2} \U^{1}\left(\mathfrak{I}_{n}\right)$ and the entries of $x$ satisfying:
\begin{align}\label{XValues}
    & x_{i, j}  \in \mathcal{P}_{\F} \, \text{ for } \, 1 \leq i \leq j \leq n-1, \\ \notag
    & x_{i, i-1}  \in a(v) \varpi_{\F}\left(1 + \mathcal{P}_{\F}\right) \, \text{ for } \, 2 \leq i \leq n, \text{ and } \\ \notag
    & x_{i, j}  \in \mathcal{P}_{\F}^{2} \, \text{ for all other entries}.
\end{align}

Hence, we may decompose $\alpha(h, x)$ as 
\begin{align*}
    \alpha(h, x) &= \begin{pmatrix}
          & 1 & \\
          &   & I_{n-1} \\
          h & \C(v, h) & x
    \end{pmatrix} \\ \notag
    &= u \, \beta_{f}^{-n} \, a(v)(cv + \sigma_{f}) \, z(h, x)
\end{align*}
where
\[
u = 
\begin{pmatrix}
I_{n} & (dv-1)(v\varpi_{\F})^{-1}\, I_{n} \\
 & I_{n}
\end{pmatrix} \in \N(2n, \F)
\]
and
\[
z(h, x)= \begin{pmatrix}
    (a(v)v \varpi_{\F}^{2})^{-1}h & (a(v)v \varpi_{\F}^{2})^{-1}\left(\C(v, h) x - a(v)\varpi_{\F} \, I_{n}\right) \\
    & I_{n}
\end{pmatrix} \in \U^{1}\left(\mathfrak{A}_{2n}\right)
\]
with $\C(v, h) x$ denoting the $n \times n$ matrix $\begin{pmatrix}
    \C(v, h) & x 
\end{pmatrix}$.~This factorization along with the fact that 
\[
|\det(h)| = q_{\F}^{-2n}
\]
for $h \in a(v)v\varpi_{\F}^{2} \U^{1}\left(\mathfrak{I}_{n}\right)$ enables us to rewrite \eqref{TildePsiEq1} as
    \begin{align}\label{Align}
    & \int\displaylimits_{\smallbackslant{\N(n, \, \F)}{\N(n, \, \F) \, a(v)v\varpi_{\F}^{2} \U^{1}\left(\mathfrak{I}_{n}\right)}}\,\int\displaylimits_{\Mat(n \times n-1, \, \F)} \mathcal{W}\left(\alpha(h, x)\right)\mathcal{W}'(h)|\det(h)|^{s-\frac{n}{2}} \, dx  \, dh \\ \notag
    &= q_{\F}^{n^{2}-2ns}\int\displaylimits_{\smallbackslant{\N(n, \, \F)}{\N(n, \, \F) \, a(v)v\varpi_{\F}^{2} \U^{1}\left(\mathfrak{I}_{n}\right)}}\,\int\displaylimits_{\Mat(n \times n-1, \, \F)} \mathcal{W}\left(\alpha(h, x)\right)\mathcal{W}'(h) \, dx  \, dh \\ \notag
    &= \zeta^{-n} \, \chi(a(v)(cv + \sigma_{f})) \, \omega_{(v,\,\varphi,\,\zeta')}\left(a(v)v \varpi_{\F}^{2}\right) \, q_{\F}^{n^{2}-2ns} \\ \notag
    & \hspace{0.3in} \cdot \int\displaylimits_{\smallbackslant{\N(n, \, \F)}{\N(n, \, \F) \, a(v)v\varpi_{\F}^{2} \U^{1}\left(\mathfrak{I}_{n}\right)}}\,\int\displaylimits_{\Mat(n \times n-1, \, \F)}  \mathcal{W}\left(z(h, x)\right)\mathcal{W}'((a(v)v\varpi_{\F}^{2})^{-1}h) \, dx  \, dh. 
\end{align}

Next, we see that 
\begin{equation}\label{WhittakerEq1}
\mathcal{W}\left(z(h, x)\right) = \psi_{\F}\left((a(v)v\varpi_{\F}^{2})^{-1}\left(h_{n, \, 1}(v\varpi_{\F})^{-1} + \sum_{i=1}^{n-1}h_{i, \, i+1} \right)\right)
\end{equation}
as $\mathcal{W}$ is equal to $\psi_{\beta_{f}}$ on $\U^{1}\left(\mathfrak{A}_{2n}\right)$.~Similarly, 
\begin{equation}\label{WhittakerEq2}
\mathcal{W}'((a(v)v\varpi_{\F}^{2})^{-1}h) = \psi_{\F}^{-1}\left((a(v)v\varpi_{\F}^{2})^{-1}\left(h_{n, \, 1}(v\varpi_{\F})^{-1} + \sum_{i=1}^{n-1}h_{i, \, i+1} \right)\right)
\end{equation}
as $\mathcal{W}'$ is equal to $\psi_{\beta_{v}}$ on $\U^{1}\left(\mathfrak{I}_{n}\right)$.

Equations \eqref{WhittakerEq1} and \eqref{WhittakerEq2}, together with \eqref{XValues}, imply that \eqref{Align} is equal to
\begin{align*}
    & \zeta^{-n} \,  \chi(a(v)(cv + \sigma_{f})) \, \omega_{(v,\,\varphi,\,\zeta')}\left(a(v)v \varpi_{\F}^{2}\right) \, q_{\F}^{n^{2}-2ns} \, \vol\left(\N\left(n, \F\right) \backslash \N\left(n, \F\right) \U^1\left(\mathfrak{I}_{n}\right)\right) \\
    & \cdot \int\displaylimits_{\mathcal{P}_{\F}} dx_{1, \, 1} \cdots \int\displaylimits_{\mathcal{P}_{\F}} dx_{n-1, \, n-1}  \int\displaylimits_{a(v)\varpi_{\F}(1+\mathcal{P}_{\F})} dx_{2, \, 1} \, \cdots \int\displaylimits_{a(v)\varpi_{\F} (1+\mathcal{P}_{\F})} dx_{n, \, n-1} \int\displaylimits_{\mathcal{P}_{\F}^{2}} dx_{n, \, 1}  \cdots \int\displaylimits_{\mathcal{P}_{\F}^{2}} dx_{n, \, n-2} \\
    &= \zeta^{-n} \, \chi(a(v)(cv + \sigma_{f})) \, \omega_{(v,\,\varphi,\,\zeta')}\left(a(v)v \varpi_{\F}^{2}\right) \, q_{\F}^{n^{2}-2ns} \, q_{\F}^{n-n^{2}} \, \vol\left(\N\left(n, \F\right) \backslash \N\left(n, \F\right) \U^1\left(\mathfrak{I}_{n}\right)\right) \\
    &= \zeta^{-n} \, \chi(a(v)(cv + \sigma_{f})) \, \omega_{(v,\,\varphi,\,\zeta')}\left(a(v)v \varpi_{\F}^{2}\right) \, q_{\F}^{n\left(1-2s \right)} \, \vol\left(\N\left(n, \F\right) \backslash \N\left(n, \F\right) \U^1\left(\mathfrak{I}_{n}\right)\right).
\end{align*}

In the second equality above, we use the right-translation invariance of $dh$ and the fact that the product of the indicated $x_{i, \, j}$-integrals contributes a factor of $q_{\F}^{n-n^{2}}$.~Indeed, there are $\frac{n(n-1)}{2}$ entries of $x$ with volume $q_{\F}^{-1/2}$ and $\frac{n(n-1)}{2}$ entries of $x$ with volume $q_{\F}^{-3/2}$.~Finally, we note that 
\[
a(v) = -\frac{1}{vf\left(v^{-1}\right)}. \qedhere
\]
\end{proof}

Equation \eqref{GammaSimple}, together with Propositions \ref{AutomorphicProposition1} and \ref{AutomorphicProposition2}, now gives us the formula for $\gamma\left(s, \, \pi_{(f, \, \chi, \, \zeta)}\times \pi_{(v, \, \varphi, \, \zeta')}, \, \psi_{\F}\right)$.
\begin{corollary}\label{CorollaryMain}
Let $\pi_{(f, \, \chi, \, \zeta)}$ be a middle supercuspidal representation of $\GL(2n, \F)$ and $\pi_{(v, \, \varphi, \, \zeta')}$ a simple supercuspidal representation of $\GL(n, \F)$.~Then
\begin{equation*}
    \gamma\left(s, \, \pi_{(f, \, \chi, \, \zeta)}\times \pi_{(v, \, \varphi, \, \zeta')}, \, \psi_{\F}\right) = \zeta^{-n} \, \chi\left(-\frac{cv + \sigma_{f}}{vf\left(v^{-1}\right)}\right) \, \omega_{(v,\,\varphi,\,\zeta')}\left(\frac{\varpi_{\F}^{2}}{f\left(v^{-1}\right)}\right) \, q_{\F}^{n(1-2s)}.
\end{equation*}
\end{corollary}

\section{Galois Side}\label{GaloisSide}

In this section, we switch our attention to the Galois side of the local Langlands correspondence.~We first recall the parametrization of irreducible $n$-dimensional representations of the Weil group $\W_{\F}$ in the essentially tame case via admissible pairs, together with the local factors attached to admissible quasi-characters and induced irreducible Weil representations.~We then compute the gamma factor of the Langlands parameter associated with a middle supercuspidal representation of $\GL(2n,\F)$ twisted by a tamely ramified quasi-character of $\F^\times$.~Finally, we compute the gamma factor attached to the tensor product of the Langlands parameters corresponding to a middle supercuspidal representation of $\GL(2n, \F)$ and a simple supercuspidal representation of $\GL(n, \F)$.~Throughout this section, we assume that $\F$ has characteristic zero and impose the condition that $p \nmid 2n$.

\subsection{Representations of the Weil group}

In this subsection, we adopt the notation used in \cite[Section 2]{Adrian-Stevens}.~We normalize the additive valuation $\val_{\F}$ on $\F$ to have image $\mathbb{Z} \, \cup \, \{ \infty \}$, and extend this to an additive valuation on a fixed separable closure $\overline{\F}$ of $\F$; thus, for any finite extension $\E/\F$, we have $\val_{\F}(\E^{\times}) = \frac{1}{e(\E|\F)}\mathbb{Z}$.~We also write $| \cdot |$ for the normalized absolute value on $\overline{\F}$, given by
\[
|x| = q_{\F}^{-\val_{\F}(x)}, 
\]
for $x \in \overline{\F}^{\times}$.~Note that if $\E/\F$ is a finite extension, then the usual normalization of the absolute value on $\E$ is given by $x \mapsto |x|^{[\E \, : \, \F]}$, for $x \in \E$. 

Next, we define an additive quasi-character $\psi_{\E}$ of $\E$ by
\[
\psi_{\E}=\psi_{\F}\circ \tr_{\E/\F}.
\]
We note that $\psi_{\E}$ is also of level one.~For $r$ a real number, we write
\[
\mathfrak{P}_{\E}^{r} = \{x \in \E \, : \, \val_{\F}(x) \geq r \} \, \, \text{ and } \, \,
\mathfrak{P}_{\E}^{r+} = \{x \in \E \, : \, \val_{\F}(x) > r \}.
\]
From these definitions, $\mathfrak{P}_{\E}^{1/e(\E|\F)} = \mathcal{P}_{\E} = \mathfrak{P}_{\E}^{0+}$.~We put $\U_{\E} = \U_{\E}^{0} = \mathcal{O}_{\E}^{\times}$, with filtration subgroups $\U_{\E}^{r} = 1 + \mathfrak{P}_{\E}^{r}$, for $r >0$, and $\U_{\E}^{r+} = 1 + \mathfrak{P}_{\E}^{r+}$, for $r \geq 0$. 

Let $\W_{\F}$ denote the Weil group of $\F$.~Implicit in the definition of the Weil group and local class field theory are compatible choices of isomorphisms $\E^{\times} \cong \W_{\E}^{\text{ab}}$, for each finite extension $\E/\F$, via which we will identify the quasi-characters of $\E^{\times}$ and of $\W_{\E}$.~Moreover, if $\F \subset \E \subset \K$ and $\xi$ is a quasi-character of $\E^{\times}$, let $\xi_{\K} := \xi \circ N_{\K/\E}$ with $N_{\K/\E}$ denoting the field norm map.~When $\xi$ is viewed as a quasi-character of $\W_{\E}$, the quasi-character $\xi_{\K}$ is simply the restriction of $\xi$ to the subgroup $\W_{\K}$.~For $\E/\F$ a finite extension, we write 
\[
\Ind_{\E/\F} = \Ind_{\W_{\E}}^{\W_{\F}}
\]
for the induction functor.

\subsection{Admissible pairs}
In our current situation, there is a nice parametrization of the irreducible $n$-dimensional representations of $\W_{\F}$ in terms of admissible quasi-characters, introduced by Howe \cite{Howe}.
\begin{definition}
    \normalfont
    An \textit{admissible pair of degree $n$} is a pair $(\E/\F, \, \xi)$ where $\E/\F$ is a (tamely ramified) degree $n$ extension and $\xi$ is a quasi-character of $\E^{\times}$ such that
    \begin{enumerate}
        \item $\xi$ does not come via the norm from a proper subfield of $\E$ containing $\F$; and
        \item if the restriction $\xi \, \big |_{\, \U_{\E}^{0+}}$ comes via the norm from a subfield $\F \subset \text{L} \subset \E$, then $\E/\text{L}$ is unramified.
    \end{enumerate}
    Two admissible pairs $(\E_{1}/\F, \xi_{1})$ and $(\E_{2}/\F, \xi_{2})$ are said to be conjugate if there is an $\F$-isomorphism from $\E_{1}$ to $\E_{2}$ which takes $\xi_{1}$ to $\xi_{2}$.
\end{definition}

In this situation we also say that $\xi$ is an \textit{admissible quasi-character} of $\E^{\times}$ (relative to $\F$).~For any such admissible quasi-character, we write
\[
\rho_{\xi} := \Ind_{\E/\F}\xi.
\]
The following theorem is from \cite[Theorem 2.2.2]{Moy}.
\begin{theorem}\label{TheoremGalois}
    If $(\E/\F, \, \xi)$ is an admissible pair of degree $n$, then $\rho_{\xi}$ is an irreducible $n$-dimensional representation of $\W_{\F}$.~Furthermore, two admissible pairs induce equivalent representations if and only if they are conjugate, and each irreducible $n$-dimensional representation of $\W_{\F}$ is induced from an admissible pair.
\end{theorem}
Let $\E/\F$ be a degree $n$ extension, and $\xi$ a \textit{ramified quasi-character} of $\E^{\times}$; that is, $\xi$ is non-trivial on $\U_{\E}$.~Let $d(\xi) \in \frac{1}{e(\E|\F)}\mathbb{Z}$ be the \textit{normalized depth} of $\xi$; that is, $\xi$ is trivial on $\U_{\E}^{d(\xi)+}$ but not on $\U_{\E}^{d(\xi)}$.~There is then $c_{\xi}\in \mathfrak{P}_{\E}^{-d(\xi)}$, well-defined modulo $\mathfrak{P}_{\E}^{-d(\xi)/2}$ such that 
\[
\xi(1+x) = \psi_{\E}(c_{\xi}x)
\]
for $x \in \mathfrak{P}_{\E}^{d(\xi)/2+}$.~We say that $c_{\xi}$ \textit{represents} $\xi$. 

\subsection{Local factors}\label{LocalFactors}
As in the preceding subsections, let $\E/\F$ be a degree $n$ extension.~Associated with any quasi-character $\xi$ of $\E^{\times}$ are local factors $L(s, \, \xi)$, $\epsilon(s, \, \xi, \, \psi_{\E})$, and $\gamma(s, \, \xi, \, \psi_{\E})$, by Tate’s thesis \cite{Tate} (see also the account in \cite[Section 23]{BH-7}).~When $\xi$ is ramified we have $L(s, \xi) = 1$ and $\epsilon(s, \, \xi, \, \psi_{\E}) = \gamma(s, \, \xi, \, \psi_{\E})$.~Denote by $\epsilon(\xi, \, \psi_{\E})$ the value of the $\epsilon$-factor at $s= 0$.~Note that $\epsilon(s, \, \xi, \, \psi_{\E}) = \epsilon(\xi|\cdot|^{ns}, \, \psi_{\E})$; the extra factor $n$ appears here because of our normalization of the absolute value.~We have the following proposition from \cite[(2.3.17)]{Moy}.
\begin{proposition}\label{PropositionGalois}
    For $\xi$ a ramified quasi-character of $\E^{\times}$ of depth $d(\xi)$ with $e(\E \, | \, \F) \, d(\xi)$ odd, we have
    \[
    \epsilon(\xi, \, \psi_{\E}) = \xi^{-1}(c_{\xi}) \, \psi_{\E}(c_{\xi}) \, |c_{\xi}|^{n/2}. 
    \]
\end{proposition}

We recall from \cite[Subsection 29.4]{BH-7}:
\begin{equation}\label{LanglandsConstant}
    \frac{\epsilon \left(s, \, \Ind_{\E/\F}\left(\mathbbm{1}_{\E^{\times}}\right), \, \psi_{\F}\right)}{\epsilon \left(s, \, \mathbbm{1}_{\E^{\times}}, \, \psi_{\E} \right)} = \lambda_{\E/\F}(\psi_{\F})
\end{equation}
where $\mathbbm{1}_{\E^{\times}}$ is the trivial quasi-character of $\W_{\E}$, and $\lambda_{\E/\F}(\psi_{\F})$ is the \textit{Langlands constant} (for more on the Langlands constant, see \cite[Subsection 30.4]{BH-7}).~Let $\F \subset \text{L} \subset \E$.~Then 
\begin{equation}\label{LanglandsConstantTower}
    \lambda_{\E/\F}\left(\psi_{\F}\right) = \lambda_{\E/\text{L}}\left(\psi_{\L}\right) \lambda_{\text{L}/\F}\left(\psi_{\F}\right)^{m}
\end{equation}
where $\left[\E : \text{L}\right] = m$ \cite[(2.4.4)]{Moy}.~Furthermore, if $\E/\F$ is a degree $n$ unramified extension, then from \cite[(2.5.3)]{Moy}, we have
\begin{equation}\label{LanglandsConstantUnramified}
    \lambda_{\E/\F}\left(\psi_{\F}\right) = (-1)^{n-1}.
\end{equation}

Next, we define the quasi-character $\varkappa_{\E/\F}$ of $\W_{\F}$ (which we also identify as a quasi-character of $\F^{\times}$ via local class field theory) by
\begin{equation}\label{eqn:varkappa}
\varkappa_{\E/\F} = \det\left(\Ind_{\E/\F}\mathbbm{1}_{\E^{\times}}\right). 
\end{equation}
We further note that $\varkappa_{\E/\F}$ has order dividing two \cite[Subsection 29.2]{BH-7}.

Lastly, we recall the following result \cite[Lemmas 2.4, 2.5]{ALST}.
\begin{lemma}\label{LemmaGalois}
Let $\E$ and {\normalfont $\text{L}$} be tamely ramified field extensions of $\F$, and let
$\chi$ and $\eta$ be quasi-characters of \, $\E^\times$ and \, {\normalfont$\text{L}^\times$},
respectively.~For $g\in \W_{\F}$, write $\E_g=g(\E)$ and
{\normalfont $\K_g=\E_g\text{L}$}, and set
\[
\theta_g := ({}^{g}\chi)_{ \, \K_g}\eta_{ \, \K_g}.
\]
Then {\normalfont
\[
\rho_\chi\otimes \rho_\eta
\cong
\bigoplus_{g \, \in \, \smalldoublecoset{\W_{\text{L}}}{\W_{\F}}{\W_{\E}}}
\Ind_{\K_g/\F}\theta_g
\]
}
and
{\normalfont
\[
\gamma(s, \, \rho_\chi\otimes \rho_\eta, \, \psi_{\F})
=
\prod_{g \, \in \, \smalldoublecoset{\W_{\text{L}}}{\W_{\F}}{\W_{\E}}}
\lambda_{\K_g/\F}(\psi_{\F}) \, \gamma(s, \, \theta_g, \, \psi_{\K_g}).
\]
}
\end{lemma}

\subsection{Twisting by tamely ramified quasi-characters}

Throughout the rest of this section, let $\Phi_{(f, \, \chi, \, \zeta)}$ be the Langlands parameter corresponding to a middle supercuspidal representation $\pi_{(f, \, \chi, \, \zeta)}$ of $\GL(2n, \F)$.~Then
\[
\Phi_{(f, \, \chi, \, \zeta)} = \Ind_{\E_{f}/\F}\xi_{(f, \, \chi, \, \zeta)}
\]
where $\left(\E_{f}/\F, \, \xi_{(f, \, \chi, \, \zeta)}\right)$ is an admissible pair of degree $2n$ \cite[Parametrization Theorem, p.~2]{BH-6}.

The following proposition describes the restriction of $\xi_{(f, \, \chi, \, \zeta)}$ to $1+\mathcal{P}_{\E_f}$, and this restriction in turn determines the representing element of $\xi_{(f, \, \chi, \, \zeta)}$.~We note that~$\E_{f}$ is naturally contained in~$\Mat(2n,\F)$, by definition. 
\begin{proposition}\label{1units}
    On $1 + \mathcal{P}_{\E_{f}}$, the quasi-character $\xi_{(f, \, \chi, \, \zeta)}$ satisfies: 
\[
   \xi_{(f, \, \chi, \, \zeta)} \, |_{1+\mathcal{P}_{\E_{f}}}
   = \psi_{\beta_{f}}|_{1+\mathcal{P}_{\E_{f}}}.
\]
\end{proposition}
\begin{proof}
    We first consider the \textit{admissible 1-pair} $\left(\E_{f}/\F, \, \xi_{(f, \, \chi, \, \zeta)} \big|_{1+ \mathcal{P}_{\E_{f}}}\right)$ as defined in \cite[Subsection 1.3]{BH-4}.~Associated with $\pi_{(f, \, \chi, \, \zeta)}$ is the simple character $\psi_{\beta_{f}}$ whose \textit{endo-class} we denote by $\Theta_{\beta_{f}}$ (for an in-depth definition of the endo-class of a simple character, see \cite{BH-2}).~Since $p \nmid 2n$, we have that $\Theta_{\beta_{f}}$ is \textit{tamely ramified} \cite{BH-2}.~Hence, the discussion following \cite[Theorem 1.3]{BH-4} and \cite[Proposition 2.3]{BH-4} imply Proposition \ref{1units}.
\end{proof}

Next, let $\nu$ be a tamely ramified quasi-character of $\F^{\times}$.~We note that
\[
\Ind_{\E_{f}/\F}(\xi_{(f, \, \chi, \, \zeta)}) \otimes \nu = \Ind_{\E_{f}/\F}(\xi_{(f, \, \chi, \, \zeta)} \otimes \nu_{\E_{f}})
\]
where $\nu_{\E_{f}} = \nu \circ N_{\E_{f}/\F}$.~Since $\xi_{(f, \, \chi, \, \zeta)} \otimes \nu_{\E_{f}}$ is ramified, we have that
\begin{equation}\label{Equation1}
        \gamma(s, \, \Phi_{(f, \, \chi, \, \zeta)} \otimes \nu, \, \psi_{\F}) = \gamma(s, \, \Ind_{\E_{f}/\F}(\xi_{(f, \, \chi, \, \zeta)} \otimes \nu_{\E_{f}}), \, \psi_{\F}).
\end{equation}
This gives us the following proposition.
\begin{proposition}\label{PropositionGaloisTamelyRamified}
$\gamma\left(s, \, \Phi_{(f, \, \chi, \, \zeta)} \otimes \nu, \, \psi_{\F} \right) 
        = \lambda_{\E_{f}/\F}(\psi_{\F}) \, \xi_{(f, \, \chi, \, \zeta)}\left(\beta_{f}^{-1}\right)  \nu(-c^{-1}\varpi_{\F}^{2}) \, q_{\F}^{1-2s}$.
\end{proposition}
\begin{proof}
    From \eqref{LanglandsConstant} and \eqref{Equation1}, we have that 
    \begin{align*}
        \gamma(s, \, \Phi_{(f, \, \chi, \, \zeta)} \otimes \nu, \, \psi_{\F}) &= \gamma(s, \, \Ind_{\E_{f}/\F}(\xi_{(f, \, \chi, \, \zeta)} \otimes \nu_{\E_{f}}), \, \psi_{\F}) \\
        &= \lambda_{\E_{f}/\F}(\psi_{\F}) \, \gamma(s, \, \xi_{(f, \, \chi, \, \zeta)} \otimes \nu_{\E_{f}}, \, \psi_{\E_{f}}).
    \end{align*}   
Next, Proposition \ref{1units} tells us that the representing element $c_{\xi_{(f, \, \chi, \, \zeta)}}$ of $\xi_{(f, \, \chi, \, \zeta)}$ is $\beta_{f}$ and $\val_{\F}\left(\beta_{f}\right) = -\sfrac{1}{n}$.

Since $\e\left(\E_{f} \, | \, \F\right)d\left(\xi_{(f, \, \chi, \, \zeta)}\right) = 1$, we have from Proposition \ref{PropositionGalois} that
\begin{align*}
    \gamma(s, \, \xi_{(f, \, \chi, \, \zeta)} \otimes \nu_{\E_{f}}, \, \psi_{\E_{f}}) &= \xi_{(f, \, \chi, \, \zeta)}\left(\beta_{f}^{-1}\right)  \nu_{\E_{f}}\left(\beta_{f}^{-1}\right)  \left|\beta_{f}\right|^{n-2ns} \\
    &= \xi_{(f, \, \chi, \, \zeta)}\left(\beta_{f}^{-1}\right) \nu(-c^{-1}\varpi_{\F}^{2}) \, q_{\F}^{1-2s}.
\end{align*}
Proposition \ref{PropositionGaloisTamelyRamified} now follows. 
\end{proof}

When we match the twisted gamma factor in Proposition \ref{AutomorphicTamelyRamified} with the gamma factor in Proposition \ref{PropositionGaloisTamelyRamified} (equal by \cite{Harris,Henniart1}), we have the following corollary.
\begin{corollary}\label{UniformizerCorollary}
    The quasi-character $\xi_{(f, \, \chi, \, \zeta)}$ satisfies: 
    \[
    \xi_{(f, \, \chi, \, \zeta)}\left(\beta_{f}\right) = \zeta \cdot \lambda_{\E_{f}/F}(\psi_{\F}).
    \]
\end{corollary}

\subsection{Twisting by simple supercuspidal representations}\label{SimpleSupercuspidalGalois}

Let $\Phi_{(v,\,\varphi,\,\zeta')}$ be the Langlands parameter corresponding to a simple supercuspidal representation $\pi_{(v,\,\varphi,\,\zeta')}$ of $\GL(n,\F)$. Then
\[
    \Phi_{(v,\,\varphi,\,\zeta')}
    =
    \Ind_{\E_{v}/\F}\eta_{(v,\,\varphi,\,\zeta')}
\]
where $\left(\E_{v}/\F, \, \eta_{(v,\,\varphi,\,\zeta')}\right)$ is an admissible pair of degree $n$ with $\eta := \eta_{(v,\,\varphi,\,\zeta')}$ being represented by $-\beta_v$ \cite[p.~129]{Adrian-Liu}.~Explicitly, $\eta$ is given by the following proposition \cite[Theorem 3.17]{Adrian-Liu}.
\begin{proposition}\label{SimpleSupercuspidalParameter}
The quasi-character $\eta$ is uniquely determined by the following:
\begin{enumerate}
    \item for $y \in 1+\mathcal{P}_{\E_v}$,
    \[
        \eta(y)
        =
        (\psi_{\beta_v}\circ \iota)(y),
    \]
    where
    \[
        \iota \, : \, \E_v^{\times} \hookrightarrow \GL(n, \F) 
    \]
    is the embedding induced by a chosen $\F$-basis of $\E_v$;

    \item on $k_F^{\times}$,
    \[
        \eta \, \big|_{ \, k_{\F}^{\times}}
        =
        \omega_{(v, \, \varphi, \, \zeta')} \, \big|_{ \, k_{\F}^{\times}}
        \otimes
        \left(\varkappa_{\E_v/\F} \, \big|_{\, k_{\F}^{\times}}\right)^{-1};
    \]

    \item at $\beta_v^{-1}$,
    \[
        \eta(\beta_v^{-1})
        =
        \left(\zeta'\cdot \lambda_{\E_v/\F}(\psi_{\F})\right)^{-1}.
    \]
\end{enumerate}
\end{proposition}

For $g \in \W_{\F}$, let
\begin{equation}\label{DegreeNotation}
    d_{g} := [g\left(\E_{f}\right) \cap \E_{v}  :  \F] \, \, \text{ and } \, \, \K_{g} := g\left(\E_{f}\right)\E_{v}.
\end{equation}
Then 
\[
[\K_{g}   :  \F] = \frac{[g\left(\E_{f}\right)   :  \F][\E_{v}   :  \F]}{[g\left(\E_{f}\right) \cap \E_{v}  :  \F]} = \frac{2n^{2}}{d_{g}}.
\]
From \cite[Theorem 3, p.~504]{Cornell}, we have the following result:
\begin{theorem}[Abhyankar's lemma]\label{Abhyankar}
    Let $\E_{1}$, $\E_{2}$, and $\F$ be local fields, $\E_{1}$, $\E_{2}$ finite extensions of $\F$.~Suppose $\E_{2}$ is tamely ramified and $\e\left(\E_{2} \, | \, \F\right) \mid \e\left(\E_{1} \, | \, \F\right)$.~Then $\E_{2}\E_{1}$ is an unramified extension of $\E_{1}$.
\end{theorem}

Theorem \ref{Abhyankar} gives rise to the following diagram that illustrates the ramification indices and residue degrees between $\K_{g}$, $g\left(\E_{f}\right)$, $\E_{v}$, and $\F$: 
\begin{equation}\label{Diagram}
\begin{tikzpicture}[scale=1.3]
  \node (Kg) at (0,2) {$\K_{g}$};
  \node (Ef) at (-2,0) {$g\left(\E_{f}\right)$};
  \node (Ev) at (2,0) {$\E_{v}$};
  \node (F)  at (0,-2) {$\F$};

  \draw (Kg) -- node[left, sloped, above] {$\e=1,\ \f=\frac{n}{d_{g}}$} (Ef);
  \draw (Kg) -- node[right, sloped, above] {$\e=1,\ \f=\frac{2n}{d_{g}}$} (Ev);

  \draw (Ef) -- node[pos=0.42, left, sloped, below] {$\e=n,\ \f=2$} (F);
  \draw (Ev) -- node[pos=0.42, right, sloped, below] {$\e=n,\ \f=1$} (F);

  \draw (Kg) -- (F)
    node[pos=0.50, fill=white, inner sep=3pt]
    {$\begin{array}{c}
        \e=n\\[3pt]
        \f=\frac{2n}{d_{g}}
      \end{array}$};
\end{tikzpicture}.
\end{equation}

Let $\xi := \xi_{(f, \, \chi, \, \zeta)}$ and 
\[
\theta_{g} :=
   \left({}^{g}\xi \circ N_{\K_{g}/g\left(\E_{f}\right)} \right) \cdot 
    \left(\eta \circ N_{\K_{g}/\E_{v}}\right).
\]
We first compute the representing element $c_{\theta_{g}}$ of $\theta_{g}$. 
\begin{proposition}\label{RepresentingElement}
    The quasi-character $\theta_{g}$ has depth $d\left(\theta_{g}\right) = \frac{1}{n}$ and is represented by
    \[
    c_{\theta_{g}} := g\left(\beta_{f}\right) - \beta_{v}.
    \]
\end{proposition}
\begin{proof}
    Since $\K_{g}/\E_{v}$ is tame, $\eta\circ N_{\K_{g}/\E_{v}}$ is represented by $-\beta_v$ and has depth $\sfrac{1}{n}$ \cite[Subsection 2.1]{Adrian-Stevens}.~Likewise, we have that ${}^{g}\xi\circ N_{\K_{g}/g\left(\E_{f}\right)}$ has the same depth and is represented by $g\left(\beta_{f}\right)$ \cite[Subsection 2.1]{Adrian-Stevens}.
    
    For $x \in \mathfrak{P}_{\K_{g}}^{1/n} = \mathcal{P}_{\K_{g}}$, we have
\[
    \left({}^{g}\xi\circ N_{\K_{g}/g\left(\E_{f}\right)}\right)(1+x)=\psi_{\K_{g}}(g\left(\beta_{f}\right) x)
\]
and
\[
    \left(\eta\circ N_{\K_{g}/\E_{v}}\right)(1+x)
    =
    \psi_{\K_{g}}(-\beta_v \, x).
\]
Since $\val_{\F}(\beta_v)=\val_{\F}(g\left(\beta_{f}\right)) = -\sfrac{1}{n}$, the quotient $u := -\frac{\beta_v}{g\left(\beta_{f}\right)}$ is a unit of $O_{\K_{g}}$.~Therefore,
\[
    \theta_{g}(1+x)
    =
    \psi_{\K_{g}}((g\left(\beta_{f}\right) - \beta_v) \, x)
    =
    \psi_{\K_{g}}\left((1 + u)g\left(\beta_{f}\right) x \right).
\]

It remains to verify that $1+u \in \mathcal O_{\K_g}^{\times}$.~For the sake of contradiction, suppose $1+u \in \mathcal{P}_{\K_{g}}$.~Then $\overline{u} = -1$ in $k_{\K_{g}}$.~Raising both sides of $u g\left(\beta_{f}\right) = -\beta_v$ to the power of $n$, we have that
\begin{equation}\label{Relation1}
    u^{n}\left(g\left(\sigma_{f}\right)\right) = (-1)^{n} \, v^{-1}.
\end{equation}
Reducing \eqref{Relation1} modulo $\mathcal{P}_{\K_{g}}$, we have that $\overline{g\left(\sigma_{f}\right)} =\overline{v}^{-1} \in k_{\F}^{\times}$.

Thus we obtain a contradiction:~$\overline{\sigma_f}$ generates a quadratic extension of $k_{\F}$, and hence so does $\overline{g(\sigma_f)}$.~Therefore, we have that $\val_{\F}(g\left(\beta_{f}\right) - \beta_v)=-\sfrac{1}{n}$ and $d(\theta_{g})=\sfrac{1}{n}$.
\end{proof}

Next, we compute the gamma factor $\gamma\left(s, \, \Phi_{(f,\,\chi,\,\zeta)}\otimes
\Phi_{(v,\,\varphi,\,\zeta')}, \, \psi_{\F}\right)$.~For the following proposition, let 
\begin{equation}\label{D(v)}
D(v) := \left| \, \bigdoublecoset{\W_{\E_v}}{\W_{\F}}{\W_{\E_f}} \, \right|
\end{equation}
where $v \in \mathcal{O}_{\F}^{\times}$.
\begin{proposition}\label{GaloisPropositionMain}
    Let $\Phi_{(f,\,\chi,\,\zeta)}$ be the Langlands parameter of $\pi_{(f, \, \chi, \, \zeta)}$ and $\Phi_{(v,\,\varphi,\,\zeta')}$ the Langlands parameter of $\pi_{(v,\,\varphi,\,\zeta')}$.~Then 
    \begin{align*}
        & \gamma\left(s, \, \Phi_{(f,\,\chi,\,\zeta)}\otimes
\Phi_{(v,\,\varphi,\,\zeta')}, \, \psi_{\F}\right) \\
&= (-1)^{n-D(v)} \, \varkappa_{\E_{v}/\F}\left(f\left(v^{-1}\right)\right) \, \zeta^{-n} \, \xi\left(-\frac{cv + \sigma_{f}}{vf\left(v^{-1}\right)}\right) \, \omega_{(v,\,\varphi,\,\zeta')}\left(\frac{\varpi_{\F}^{2}}{f\left(v^{-1}\right)}\right) \, q_{\F}^{n(1-2s)}.
    \end{align*}
\end{proposition}
\begin{proof}
    From Propositions \ref{RepresentingElement} and \ref{PropositionGalois}, we have that
    \begin{equation}\label{MiddleSimple}
        \gamma\left(s, \, \Phi_{(f,\,\chi,\,\zeta)}\otimes
\Phi_{(v,\,\varphi,\,\zeta')}, \, \psi_{\F}\right) = \prod_{\smalldoublecoset{\W_{\E_v}}{\W_{\F}}{\W_{\E_f}}} \lambda_{\K_{g_{i}}/\F}\left(\psi_{\F}\right)\gamma\left(s, \, \theta_{g_{i}}, \, \psi_{\K_{g_{i}}}\right)
    \end{equation}
    where $\{g_{i}\}$ is a set of representatives of the double coset space $\W_{\E_{v}}\backslash \W_{\F}/\W_{\E_{f}}$.~For each $\gamma\left(s, \, \theta_{g_{i}}, \, \psi_{\K_{g_{i}}}\right)$, Proposition \ref{PropositionGalois} shows that the factor is of the form
    \begin{equation}\label{2}
    \theta_{g_{i}}^{-1}\left(c_{\theta_{g_{i}}}\right) \, |c_{\theta_{g_{i}}}|^{[\K_{g_{i}} \, : \, \F]\left(\frac{1}{2}-s\right)}
    \end{equation}
    as $\e\left(\K_{g_{i}} \, | \, \F\right) d\left(\theta_{g_{i}}\right) = 1$ from Theorem \ref{Abhyankar} and the trace of $c_{\theta_{g_{i}}}$ is zero.
    
    Recall the notation introduced in \eqref{DegreeNotation}.~Simplifying \eqref{2} further, we note that 
    \begin{equation}\label{RepresentingElement0}
            |c_{\theta_{g_{i}}}|^{[\K_{g_{i}} \, : \, \F]\left(\frac{1}{2}-s\right)} = q_{\F}^{\frac{2n}{d_{g_{i}}}\left(\frac{1}{2}-s\right)} 
    \end{equation}
    and 
    \begin{equation}\label{Norms}
        \theta_{g_{i}}^{-1}\left(c_{\theta_{g_{i}}}\right) = \left({}^{g_{i}}\xi\right)^{-1}\left(N_{\K_{g_{i}}/g_{i}\left(\E_{f}\right)} \left(c_{\theta_{g_{i}}}\right)\right) \cdot 
   \eta^{-1}\left( N_{\K_{g_{i}}/\E_{v}}\left(c_{\theta_{g_{i}}}\right)\right).
    \end{equation}
    Furthermore, we have from comparing dimensions on both sides of \eqref{MiddleSimple} that 
    \[
            \displaystyle\sum_{g_{i}} \frac{2n^{2}}{d_{g_{i}}} = 2n \cdot n = 2n^{2}.
    \]
    Dividing by $n$, we get
    \begin{equation}\label{DimensionCount}
        \displaystyle\sum_{g_{i}} \frac{2n}{d_{g_{i}}} = 2n.
    \end{equation}
Equations \eqref{RepresentingElement0}, \eqref{Norms}, and \eqref{DimensionCount} enable us to rewrite \eqref{MiddleSimple} as
\begin{align}\label{Rewrite1}
        & q_{\F}^{n-2ns} \, \prod_{\smalldoublecoset{\W_{\E_v}}{\W_{\F}}{\W_{\E_f}}}\lambda_{\K_{g_{i}}/\F}\left(\psi_{\F}\right) \left({}^{g_{i}}\xi\right)^{-1}\left(N_{\K_{g_{i}}/g_{i}\left(\E_{f}\right)} \left(c_{\theta_{g_{i}}}\right)\right) \cdot 
   \eta^{-1}\left( N_{\K_{g_{i}}/\E_{v}}\left(c_{\theta_{g_{i}}}\right)\right) \\ \notag
        &= q_{\F}^{n(1-2s)} \, \left(\prod_{\smalldoublecoset{\W_{\E_v}}{\W_{\F}}{\W_{\E_f}}}\lambda_{\K_{g_{i}}/\F}\left(\psi_{\F}\right)\right) \xi^{-1}\left(\prod_{\smalldoublecoset{\W_{\E_v}}{\W_{\F}}{\W_{\E_f}}} g_{i}^{-1}\left(N_{\K_{g_{i}}/g_{i}\left(\E_{f}\right)} \left(c_{\theta_{g_{i}}}\right)\right) \right) \\ \notag
        & \hspace{2in} \cdot \eta^{-1}\left(\prod_{\smalldoublecoset{\W_{\E_v}}{\W_{\F}}{\W_{\E_f}}} N_{\K_{g_{i}}/\E_{v}} \left(c_{\theta_{g_{i}}}\right) \right).
\end{align}

By Theorem \ref{Abhyankar}, $\K_{g_i}$ is finite and unramified over both $g_i\left(\E_f \right)$ and $\E_v$; hence the extensions $\K_{g_i}/g_i\left(\E_f \right)$ and $\K_{g_i}/\E_v$ are Galois.~Using \cite[Subsection 2.3, p.~10]{Adrian-Stevens}, we can simplify \eqref{Rewrite1} as
    \begin{align}\label{Rewrite2} 
        & q_{\F}^{n(1-2s)} \, \left(\prod_{\smalldoublecoset{\W_{\E_v}}{\W_{\F}}{\W_{\E_f}}}\lambda_{\K_{g_{i}}/\F}\left(\psi_{\F}\right)\right) \xi^{-1}\left(\prod_{\smalldoublecoset{\W_{\E_v}}{\W_{\F}}{\W_{\E_f}}} \, \prod_{\sigma \in \Gal\left(\K_{g_{i}}/g_{i}\left(\E_{f}\right)\right)} \left(\beta_{f} - g_{i}^{-1}\left(\sigma\left(\beta_{v}\right)\right)\right) \right) \\ \notag
        & \hspace{2in} \cdot \eta^{-1}\left((-1)^{2n} \, \prod_{\smalldoublecoset{\W_{\E_v}}{\W_{\F}}{\W_{\E_f}}} \, \prod_{\tau \in \Gal\left(\K_{g_{i}}/\E_{v}\right)}\left(\beta_{v} - \tau\left(g_{i}\left(\beta_{f}\right)\right)\right) \right) \\ \notag
        &= q_{\F}^{n(1-2s)} \, \left(\prod_{\smalldoublecoset{\W_{\E_v}}{\W_{\F}}{\W_{\E_f}}}\lambda_{\K_{g_{i}}/\F}\left(\psi_{\F}\right)\right)  \xi^{-1}\left(m_{\beta_{v}}\left(\beta_{f}\right)\right) \, \eta^{-1}\left(m_{\beta_{f}}\left(\beta_{v}\right)\right)
    \end{align}
    where $\Gal$ denotes the Galois group.~Using \eqref{LanglandsConstantTower}, \eqref{LanglandsConstantUnramified}, and \eqref{D(v)}, we can further rewrite \eqref{Rewrite2} as
    \begin{equation}\label{Rewrite3}
        (-1)^{n-D(v)} \, q_{\F}^{n(1-2s)} \,  \left(\prod_{\smalldoublecoset{\W_{\E_v}}{\W_{\F}}{\W_{\E_f}}}\lambda_{g_{i}\left(\E_{f}\right)/\F}\left(\psi_{\F}\right)^{\frac{n}{d_{g_{i}}}}\right)\xi^{-1}\left(m_{\beta_{v}}\left(\beta_{f}\right)\right) \, \eta^{-1}\left(m_{\beta_{f}}\left(\beta_{v}\right)\right).
    \end{equation}

From \cite[Subsection 28.7, Remark]{BH-7} and \eqref{LanglandsConstant}, we have that 
\begin{equation}\label{LanglandsConstant2}
\lambda_{g_{i}\left(\E_{f}\right)/\F}\left(\psi_{\F}\right) = \lambda_{\E_{f}/\F}\left(\psi_{\F}\right).
\end{equation}
Next, we have from \eqref{MinimalPolynomialSimple} and \eqref{MinimalPolynomialMiddle} that
\begin{align}\label{MinimalPolynomial1}
    m_{\beta_{v}}(\beta_{f})^{-1} &= \left(\beta_{f}^{n} - (v\varpi_{\F})^{-1}\right)^{-1} \\ \notag
    &= \frac{\varpi_{F}}{\sigma_{f}-v^{-1}} \\ \notag
    &= -\beta_{f}^{-n}\left(\frac{cv + \sigma_{f}}{vf\left(v^{-1}\right)}\right)
\end{align}
and
\begin{align}\label{MinimalPolynomial2}
    m_{\beta_{f}}\left(\beta_{v}\right)^{-1} &= \left(\beta_{v}^{2n} - d\varpi_{\F}^{-1}\beta_{v}^{n} - c\varpi_{\F}^{-2}\right)^{-1} \\ \notag
    &= \frac{\varpi_{\F}^{2}}{v^{-2} - dv^{-1} - c} \\ \notag
    &= \frac{\varpi_{\F}^{2}}{f\left(v^{-1}\right)}.
\end{align}
Using \eqref{DimensionCount}, \eqref{LanglandsConstant2}, \eqref{MinimalPolynomial1}, \eqref{MinimalPolynomial2}, and Corollary \ref{UniformizerCorollary}, we may simplify \eqref{Rewrite3} as
\begin{align*}
& (-1)^{n-D(v)} \, q_{\F}^{n(1-2s)} \,  \lambda_{\E_{f}/\F}\left(\psi_{\F}\right)^{n} \, \xi\left(-\beta_{f}^{-n}\left(\frac{cv + \sigma_{f}}{vf\left(v^{-1}\right)}\right)\right) \, \eta\left(\frac{\varpi_{\F}^{2}}{f\left(v^{-1}\right)}\right) \\
&= (-1)^{n-D(v)} \, q_{\F}^{n(1-2s)}  \, \zeta^{-n} \,  \xi\left(-\frac{cv + \sigma_{f}}{vf\left(v^{-1}\right)}\right) \, \eta\left(\frac{\varpi_{\F}^{2}}{f\left(v^{-1}\right)}\right).
\end{align*}

Lastly, \cite[Proposition 29.2]{BH-7} and \cite[p.~2]{Harris} tell us that 
\[
\eta\left(\frac{\varpi_{\F}^{2}}{f\left(v^{-1}\right)}\right) = \varkappa_{\E_{v}/\F}\left(f\left(v^{-1}\right)\right) \,\omega_{(v,\,\varphi,\,\zeta')}\left(\frac{\varpi_{\F}^{2}}{f\left(v^{-1}\right)}\right)
\]
which completes the proof of Proposition \ref{GaloisPropositionMain}.
\end{proof}

\subsection{Langlands parameter}\label{LanglandsParameter}

In this subsection, we give the Langlands parameter  $\Phi_{(f,\,\chi,\,\zeta)}$ of an essentially tame middle supercuspidal representation of $\GL(2n, \F)$.~To do so, we explicitly describe the quasi-character $\xi_{(f,\,\chi,\,\zeta)}$ on $\E_{f}^{\times}$. 

We first decompose $\E_{f}^{\times}$ as 
\[
\E_{f}^{\times} = \left\langle \beta_{f} \right\rangle \times \mu_{\normalfont{\text{L}_{f}}}' \times \left(1 + \mathcal{P}_{\E_{f}}\right).
\]
At $\beta_{f}$, the quasi-character $\xi_{(f, \, \chi, \, \zeta)}$ is given by Corollary \ref{UniformizerCorollary}.~Moreover, the restriction of $\xi_{(f, \, \chi, \, \zeta)}$ to $1+\mathcal{P}_{\E_f}$ is given by Proposition \ref{1units}.

We recall the following result from \cite[Proposition 29.2]{BH-7} and \cite[p.~2]{Harris}:
\begin{equation}\label{CentralCharacterCondition}
    \det\left(\Phi_{(f, \, \chi, \, \zeta)}\right) = \omega_{(f, \, \chi, \, \zeta)} = \xi_{(f, \, \chi, \, \zeta)} \, \big |_{\, \F^{\times}} \otimes \varkappa_{\E_{f}/\F} \, .
\end{equation}
With \eqref{CentralCharacterCondition} and \eqref{D(v)}, we have the following proposition.~Let $v(u):= (cu)^{-1}$.
\begin{proposition}\label{klF}
    For $u \in \mathcal{O}_{\F}^{\times}$, 
    \[
    \xi_{(f, \, \chi, \, \zeta)}(1 + u\sigma_{f}) = 
    (-1)^{D(v(u))-n} \, \varkappa_{\E_f/\F}(-cf(cu)) \, \varkappa_{\E_{v(u)}/\F}(f(cu))\, \chi(1 + u\sigma_{f}).
    \]
\end{proposition}
\begin{proof}
    Matching the gamma factors in Corollary \ref{CorollaryMain} and Proposition \ref{GaloisPropositionMain}, which are equal by \cite{Harris,Henniart1}, tells us that
    \[
    \xi_{(f, \, \chi, \, \zeta)}\left(-\frac{cv(u) + \sigma_{f}}{v(u)f\left(v(u)^{-1}\right)}\right) = (-1)^{D(v(u))-n} \, \varkappa_{\E_{v(u)}/\F}\left(f\left(v(u)^{-1}\right)\right) \, \chi\left(-\frac{cv(u) + \sigma_{f}}{v(u)f\left(v(u)^{-1}\right)}\right).
    \]
Since~$f(v(u)^{-1})=f(cu)\in\F$, Proposition~\ref{klF} now follows from~\eqref{CentralCharacterCondition} and the fact that~$\varkappa_{\E_f/\F}$ is at most quadratic. 
\end{proof}

\begin{remark}
    \normalfont 
At this stage, it is not obvious that the conjugacy class of the admissible pair $\left(\E_f/\F, \, \xi_{(f, \, \chi, \, \zeta)} \right)$ depends only on $\bar{f}$, and not on the choice of lift $f$, even though it must do from the fact that the Langlands correspondence is well-defined.~After simplification (using the appendix), this becomes clear.
\end{remark}
\begin{proof}[Proof of Theorem \ref{TheoremMain}]
    Firstly, (1) follows from Corollary~\ref{UniformizerCorollary} (note that the uniformizer in the statement of Theorem~\ref{TheoremMain} is~$\beta_f^{-1}$).~Similarly, (3) follows from Proposition~\ref{1units}, on noting that~$\tr_{\Mat(2n \times 2n, \, \F)/\F}$ coincides with~$\tr_{\E_f/\F}$ on~$\E_f$.
    
    For~(2), we note first that~\eqref{CentralCharacterCondition}, together with the fact that~$\varkappa_{\E_{f}/{\normalfont \text{L}{f}}}$ is trivial on $\mathcal{O}_{\F}^{\times}$, implies that it is enough to check the desired equality on a set of representatives for~$\mu'_{\L_f}/\mu'_{\F}$, which we take to be~$\{\sigma_{f}, \, 1+u\sigma_{f} : u\in\mu'_{\F}\}$. In Proposition~\ref{PropositionAppendix2} below, we show that, for $u\in\mathcal{O}_{\F}^{\times}$ we have
\[
\left(-1\right)^{D(v(u))-n}
\varkappa_{\E_{v(u)}/\F}\left(f(cu)\right)
=
\varkappa_{\E_f/\L_f}
\left(1+u\sigma_f\right).
\]
Moreover, in Lemma~\ref{LemmaAppendix4}, we see that~$\varkappa_{\E_f/\F}$ is unramified.~Substituting these into Proposition~\ref{klF} gives the desired formula
\[
\xi_{(f, \, \chi, \, \zeta)}(1 + u\sigma_{f}) = \varkappa_{\E_{f}/{\normalfont \text{L}_{f}}}(1 + u\sigma_{f}) \, \chi(1 + u\sigma_{f}).
\]
It remains only to check for the representative~$\sigma_{f}$, which is done in Proposition~\ref{PropositionAppendix3}, thus completing the proof of Theorem \ref{TheoremMain}.
\end{proof}

\appendix
\section{}

Let $p$ be odd and $\overline{\varepsilon}_{\F}$ denote the unique non-trivial quadratic character of $k_{\F}^{\times}$, given by
\[
\overline{\varepsilon}_{\F}(a)
=
a^{(q_{\F}-1)/2}.
\]
We denote its inflation to $\mathcal{O}_{\F}^{\times}$ by $\varepsilon_{\F}$ where
\[
\varepsilon_{\F}(x) = \overline{\varepsilon}_{\F}(\overline{x})
\]
with $\overline{x}$ being the image of $x$ in $k_{\F}^{\times}$.

Observe that the quadratic characters are compatible with norm maps along unramified extensions.~More precisely, let $\text{L}/\F$ be unramified.~Then
\[
\overline{\varepsilon}_{\F}\circ N_{k_{\L}/k_{\F}}
=
\overline{\varepsilon}_{\L}.
\]
Indeed, for $a\in k_{\L}^{\times}$, we have
\[
\overline{\varepsilon}_{\F}
\left(N_{k_{\L}/k_{\F}}(a)\right) = 
\left(
a^{(q_{\L}-1)/(q_{\F}-1)}
\right)^{(q_{\F}-1)/2} = \overline{\varepsilon}_{\L}(a).
\]
Since the norm map is compatible with reduction modulo the maximal ideal,
it follows that
\begin{equation}\label{EquationAppendix0}
\varepsilon_{\F}\circ N_{\L/\F}
=
\varepsilon_{\L}
\end{equation}
as characters of $\mathcal{O}_{\L}^{\times}$. 

For $\E/\F$ a finite tame extension, the following lemma gives some properties of $\varkappa_{\E/\F}$.
\begin{lemma}\label{LemmaAppendix1}
    Let $\E/\F$ be a finite tame extension.
    \begin{enumerate}
    \item The Langlands constant $\lambda_{\E/\F}(\psi_{\F})$ is a fourth root of unity such that
        \[
            \lambda_{\E/\F}(\psi_{\F})^{2} = \varkappa_{\E/\F}(-1).
        \]
    \item Let $\F \subset \L \subset \E$ be a tower of extensions with $[\E: \L] = m$.~Then
        \[
             \varkappa_{\E/\F} = \varkappa_{\L/\F}^{m} \cdot \left(\varkappa_{\E/\L} \, \big |_{\, \F^{\times}}\right).
        \]
    \item Let $\E/\F$ be a totally ramified extension of degree $n$.~Then
        \[
            \varkappa_{\E/\F} \, \big|_{\, \mathcal{O}_{\F}^{\times}}
            = \epsilon_{\F}^{n-1} . 
        \]
    \end{enumerate}
\end{lemma}
\begin{proof}
    Parts (1) and (2) of Lemma \ref{LemmaAppendix1} follow from \cite[(30.4.3), p.~195]{BH-7} and \cite[(10.1.3)]{BF} respectively.

    To prove (3), we first suppose that $n$ is odd.~By \cite[Proposition 10.1.6 (i)]{BF}, $\varkappa_{\E/\F}$ is unramified.~Next, suppose that $n$ is even.~By
    \cite[Proposition 10.1.6 (ii)]{BF}, there is a unique
    intermediate field
        \[
            \F \subset \K \subset \E 
        \]
    such that $\E/\K$ is quadratic, and
        \[
            \varkappa_{\E/\F}
            =
            \varkappa_{\E/\K} \, \big|_{\, \F^{\times}}.
        \]

    From \cite[Subsection 34.3]{BH-7}, we have that $\varkappa_{\E/\K}$ is a non-trivial ramified
    quadratic character of level-zero.~Consequently, its restriction to $\mathcal{O}_{\K}^{\times}$ is equal to  $\varepsilon_{\K}$.~Because $k_{\K} = k_{\F}$, the restriction of $\varepsilon_{\K}$ to $\mathcal{O}_{\F}^{\times}$ is precisely $\varepsilon_{\F}$.~Restricting $\varkappa_{\E/\F}$ to $\mathcal{O}_{\F}^{\times}$ therefore gives
        \[
            \varkappa_{\E/\F} \, \big|_{\, \mathcal{O}_{\F}^{\times}}
            =
            \varepsilon_{\F}.
        \]
    This proves (3).
\end{proof}

The following proposition translates the relevant double-coset count into a question about the factorization of a minimal polynomial.

\begin{proposition}\label{PropositionAppendix1}
Let $\overline{\F}$ be a separable closure of $\F$, and let
$\L/\F$ and $\E/\F$ be finite separable extensions contained in $\overline{\F}$.~Furthermore, suppose there exists $\beta \in \E$ such that $\E = \F[\beta]$, and let $m_{\beta} \in \F[X]$ denote the minimal polynomial of $\beta$ over $\F$.~Then 
\[
\left| \, \bigdoublecoset{\W_{\L}}{\W_{\F}}{\W_{\E}} \, \right| = \#\left\{ \, 
\text{irreducible factors of $m_{\beta}$ over $\L$} \,  \right\}.
\]
\end{proposition}
\begin{proof}
Set $\Omega_{\F} := \Gal\left(\overline{\F}/\F\right)$, $\Omega_{\E}:=\Gal\left(\overline{\F}/\E\right)$, and $\Omega_{\L}:=\Gal\left(\overline{\F}/\L\right)$.~From \cite[Proposition~2.1(b)]{Milne}, we have the following bijection:
\[
\begin{array}{rcl}
\operatorname{Hom}_{\F}\left(\E, \, \overline{\F}\right)
& \longleftrightarrow &
\left\{ \, 
\text{roots of $m_{\beta}$ in $\overline{\F}$} \, 
\right\}
\\[4pt]
\tau
& \mapsto &
\tau(\beta).
\end{array}
\]
On the other hand, \cite[Proposition~7.4]{Milne} tells us that every element of $\operatorname{Hom}_{\F}\left(\E, \, \overline{\F}\right)$ extends to an element of $\Omega_{\F}$.~Since two elements of $\Omega_{\F}$ have the same restriction to $\E$ if and only if they lie in the same right coset of $\Omega_{\E}$, the restriction map induces a bijection
\[
\begin{array}{rcl}
\bigslant{\Omega_{\F}}{\Omega_{\E}}
& \longleftrightarrow &
\operatorname{Hom}_{\F}\left(\E,\overline{\F}\right)
\\[4pt]
\sigma \, \Omega_{\E}
& \mapsto &
\sigma\big|_{\E}.
\end{array}
\]

Using \cite[Proposition~28.5(1)(c)]{BH-7}, we see that we also have a bijection
\[
\bigslant{\W_{\F}}{\W_{\E}}
\longleftrightarrow
\bigslant{\Omega_{\F}}{\Omega_{\E}}.
\]
This bijection is equivariant with respect to the natural left action of $\W_{\L}$.~Putting these together, we see that the set of double cosets
\[
\bigdoublecoset{\W_{\L}}{\W_{\F}}{\W_{\E}}
\]
is identified with the set of $\W_{\L}$-orbits on the roots of $m_{\beta}$.

Let $\K$ be the splitting field of $m_{\beta}$ over
$\L$.~Since $\E/\F$ is separable, the polynomial $m_{\beta}$ is separable.~It follows that $\K/\L$ is a finite Galois extension.~From \cite[Proposition~7.7]{Milne}, we have that the action of $\Omega_{\L}$ on the roots of $m_{\beta}$ factors through the continuous surjective restriction map
\[
\begin{array}{rcl}
\Omega_{\L}
& \longrightarrow &
\Gal\left(\K/\L\right)
\\[4pt]
\sigma
& \mapsto &
\sigma\big|_{\K}.
\end{array}
\]

By \cite[Section~28.4]{BH-7}, the subgroup $\W_{\L}$ is dense in $\Omega_{\L}$.~Its image under the restriction map above is therefore dense in the finite discrete group $\Gal\left(\K/\L\right)$, and hence is equal to the whole group.~Consequently, $\W_{\L}$ and $\Omega_{\L}$ induce the same permutations of the roots of $m_{\beta}$ and therefore have the same orbits.

Finally, by \cite[Proposition~4.26]{Milne}, the Galois orbits on the roots of a separable polynomial are precisely the sets of roots of its irreducible factors.~Choosing each irreducible factor to be monic, we conclude that
\[
\left| \, \bigdoublecoset{\W_{\L}}{\W_{\F}}{\W_{\E}} \, \right| = \#\left\{ \, 
\text{irreducible factors of $m_{\beta}$ over $\L$} \,  \right\}.  \qedhere
\]
\end{proof}

For a polynomial $f \in k_{\F}[X]$, let $\disc(f)$ denote the discriminant of $f$.~We shall use the following two facts.
\begin{lemma}\label{LemmaAppendix2}
Let $f \in k_{\F}[X]$ be a separable polynomial of degree $n$.
\begin{enumerate}
\item Suppose $X \nmid f$.~For every $v \in k_{\F}^{\times}$ and $N \geq 1$ coprime to $p$, the polynomial $f\left(v X^{N}\right)$ is separable over $k_{\F}$.

\item\label{LemmaAppendix2.2} Suppose that $p$ is odd.~Let $r$ be the number of monic irreducible factors of $f$ over $k_{\F}$.~Then
\[
\overline{\varepsilon}_{\F}
\left(\disc(f)\right)
=
\left(-1\right)^{n-r}.
\]
\end{enumerate}
\end{lemma}
\begin{proof}
The proof of (1) follows from a straightforward computation, while (2) follows from \cite[Corollary~1]{Swan}.
\end{proof}

For later use, we record a discriminant formula for polynomials of the following form.
\begin{lemma}\label{LemmaAppendix3}
Let $R$ be a commutative ring and
\[
P(Y)=Y^{2n}+AY^{n}+B,
\]
with $A, B \in R$.~Then
\[
\disc(P)
=
n^{2n}B^{n-1}\left(A^{2}-4B\right)^{n}.
\]
\end{lemma}

\begin{proof}
Differentiating $P$ gives
\[
P'(Y)
=
nY^{n-1}\left(2Y^{n}+A\right).
\]
The resultant formula then gives
\[
\Res\left(P,P'\right)
=
n^{2n}B^{n-1}\left(4B-A^{2}\right)^{n}.
\]
Since $P$ is monic of degree $2n$, it follows from
\cite[Section~3.3.2]{Cohen93} that
\begin{align*}
\disc(P)
&=
\left(-1\right)^{n}
\Res\left(P,P'\right) \\
&=
\left(-1\right)^{n}
n^{2n}B^{n-1}\left(4B-A^{2}\right)^{n} \\
&=
n^{2n}B^{n-1}\left(A^{2}-4B\right)^{n}. \qedhere
\end{align*}
\end{proof}

Using the preceding results, we now establish identities involving the quasi-characters $\varkappa_{\E_f/\F}$, $\varkappa_{\E_{v(u)}/\F}$, and $\varkappa_{\E_f/\L_f}$.
\begin{lemma}\label{LemmaAppendix4}
    The quasi-character $\varkappa_{\E_f/\F}$ is unramified.
\end{lemma}
\begin{proof}
Recall that $\L_f/\F$ is a tame unramified quadratic extension and that $\E_f/\L_f$ is a tame totally ramified extension of degree $n$.~By (2) of Lemma~\ref{LemmaAppendix1}, we have
\[
\varkappa_{\E_f/\F} =
\varkappa_{\L_f/\F}^{n} \cdot 
\left(
\varkappa_{\E_f/\L_f} \,
\big|_{ \, \F^{\times}}
\right).
\]
Since $\L_f/\F$ is unramified, it follows from
\cite[Proposition 10.1.5]{BF} that
$\varkappa_{\L_f/\F}$ is unramified.

It therefore remains to show that $\varkappa_{\E_f/\L_f}$ is trivial on $\mathcal{O}_{\F}^{\times}$.~By (3) of
Lemma~\ref{LemmaAppendix1},
\[
\varkappa_{\E_f/\L_f}
\, \big|_{ \, \mathcal{O}_{\L_f}^{\times}}
=
\varepsilon_{\L_f}^{\,n-1}.
\]
Thus, for $x\in\mathcal{O}_{\F}^{\times}$,
\[
\varepsilon_{\L_f}(x) =
\overline{x}^{\left(q_{\L_f}-1\right)/2} =
\left(\overline{x}^{q_{\F}-1}
\right)^{\left(q_{\F}+1\right)/2} = 1.
\]
This proves Lemma \ref{LemmaAppendix4}.
\end{proof}

Recall from \eqref{D(v)} that 
\[
D(v(u)) := \left| \, \bigdoublecoset{\W_{\E_{v(u)}}}{\W_{\F}}{\W_{\E_{f}}} \,\right|.
\]
The following lemma relates $D(v(u))$ and $\varkappa_{\E_{v(u)}/\F}$.
\begin{lemma}\label{LemmaAppendix5}
For $u\in\mathcal{O}_{\F}^{\times}$, 
    \[
        (-1)^{D(v(u))-n} = \varkappa_{\E_{v(u)}/\F}(-c).
    \]
\end{lemma}
\begin{proof}
By a direct computation, we see that $\E_f/\F$ is separable.~From Proposition~\ref{PropositionAppendix1}, we have
\[
D(v(u))
=
\#\left\{ \,
\text{irreducible factors of $m_{\beta_f}$ over
$\E_{v(u)}$}
\, \right\}.
\]

Substituting $X=\beta_{v(u)}Y$ into $m_{\beta_f}(X)$ and multiplying by $v(u)^{2}\varpi_{\F}^{2}$ gives
\[
P_{v(u)}(Y) := v(u)^{2}  \varpi_{\F}^{2} \, m_{\beta_f}\left(\beta_{v(u)}Y\right) = Y^{2n}-dv(u)Y^{n}-cv(u)^{2}.
\]
Since $\beta_{v(u)}\in\E_{v(u)}^{\times}$, the polynomials $P_{v(u)}$ and $m_{\beta_f}$ have the same number of irreducible factors over $\E_{v(u)}$.

Reducing $P_{v(u)}$ modulo $\mathcal{P}_{\E_{v(u)}}$ gives
\[
\overline{P}_{v(u)}(Y) = Y^{2n}-\overline{dv(u)}Y^{n}-\overline{cv(u)^{2}} \in k_{\F}[Y]
\]
as $k_{\E_{v(u)}} = k_{\F}$.~Since~$p\nmid 2n$, the polynomial $\overline{P}_{v(u)}$ is separable over $k_{\F}$.~Hence, Hensel's lemma tells us that the monic irreducible factors of $\overline{P}_{v(u)}$ over $k_{\F}$ lift uniquely to the monic irreducible factors of $P_{v(u)}$ over $\E_{v(u)}$.

It follows that
\begin{equation}\label{EquationAppendix1}
D(v(u))
=
\#\left\{ \,
\text{irreducible factors of $\overline{P}_{v(u)}$ over
$k_{\F}$}
\, \right\}.
\end{equation}
Applying (2) of Lemma~\ref{LemmaAppendix2} together with \eqref{EquationAppendix1}, we obtain
\begin{equation}\label{EquationAppendix2}
\overline{\varepsilon}_{\F}
\left(
\disc\left(\overline{P}_{v(u)}\right)
\right)
=
\left(-1\right)^{2n-D(v(u))}
=
\left(-1\right)^{D(v(u))}.
\end{equation}

Next, we set $\delta:= \overline{d^{2}+4c}$.~Using Lemma \ref{LemmaAppendix3} with $A= -\overline{dv(u)}$ and $B= -\overline{cv(u)^{2}}$, we obtain
\begin{align}
\disc\left(\overline{P}_{v(u)}\right)
&=
n^{2n}
\left(-\overline{cv(u)^{2}}\right)^{n-1}
\left(
\overline{d^{2}v(u)^{2}+4cv(u)^{2}}
\right)^{n} \notag\\
&=
n^{2n}
\left(-\overline{c}\right)^{n-1}
\overline{v(u)}^{ \, 4n-2}
\delta^{n}.
\label{EquationAppendix3}
\end{align}

The elements $n^{2n}$ and $\overline{v(u)}^{ \, 4n-2}$ are squares in $k_{\F}^{\times}$.~Moreover, $\delta$ is not a square in $k_{\F}^{\times}$ because $\bar{f}$ is irreducible over $k_{\F}$ and $p$ is odd (by applying Lemma~\ref{LemmaAppendix2}\eqref{LemmaAppendix2.2}).~Hence,
\[
\overline{\varepsilon}_{\F}(\delta)=-1.
\]
Combining this fact with \eqref{EquationAppendix2} and
\eqref{EquationAppendix3}, we have
\begin{align*}
\left(-1\right)^{D(v(u))}
&=
\overline{\varepsilon}_{\F}
\left(
\left(-c\right)^{n-1}\delta^{n}
\right)\\
&=
\varepsilon_{\F}(-c)^{n-1}
\left(-1\right)^{n}.
\end{align*}
Consequently,
\begin{equation}\label{EquationAppendix4}
\left(-1\right)^{D(v(u))-n}
=
\varepsilon_{\F}(-c)^{n-1}.
\end{equation}

From (3) of Lemma~\ref{LemmaAppendix1}, we have
\[
\varkappa_{\E_{v(u)}/\F}(-c)
=
\varepsilon_{\F}(-c)^{n-1}.
\]
The result now follows from \eqref{EquationAppendix4}.
\end{proof}

This leads us to one of the main identities of this appendix.
\begin{proposition}\label{PropositionAppendix2}
For $u\in\mathcal{O}_{\F}^{\times}$, we have
\[
\left(-1\right)^{D(v(u))-n}
\varkappa_{\E_{v(u)}/\F}\left(f(cu)\right)
=
\varkappa_{\E_f/\L_f}
\left(1+u\sigma_f\right).
\]
\end{proposition}

\begin{proof}
From Lemma~\ref{LemmaAppendix5}, we have
\begin{align}
\left(-1\right)^{D(v(u))-n}
\varkappa_{\E_{v(u)}/\F}\left(f(cu)\right)
&=
\varkappa_{\E_{v(u)}/\F}(-c)
\varkappa_{\E_{v(u)}/\F}\left(f(cu)\right)
\notag\\
&=
\varkappa_{\E_{v(u)}/\F}
\left(-cf(cu)\right).
\label{EquationAppendix5}
\end{align}
Next, we note that the roots of $f$ are $\sigma_f$ and $d-\sigma_f$.~Since $\L_f/\F$ is Galois, it follows that
\begin{align}
N_{\L_f/\F}
\left(1+u\sigma_f\right)
&=
\left(1+u\sigma_f\right)
\left(1+u\left(d-\sigma_f\right)\right)
\notag\\
&=
1+u\left(
\sigma_f+d-\sigma_f
\right)
+
u^{2}\sigma_f\left(d-\sigma_f\right)
\notag\\
&=
1+du-cu^{2}.
\label{EquationAppendix6}
\end{align}

On the other hand,
\[
f(cu)
=
c^{2}u^{2}-cdu-c
=
-c\left(1+du-cu^{2}\right).
\]
Combining this identity with \eqref{EquationAppendix6}, we obtain
\begin{equation}\label{EquationAppendix7}
-cf(cu)
=
c^{2}
N_{\L_f/\F}
\left(1+u\sigma_f\right).
\end{equation}
Hence, \eqref{EquationAppendix5} and \eqref{EquationAppendix7} give
\begin{equation}\label{EquationAppendix8}
\left(-1\right)^{D(v(u))-n}
\varkappa_{\E_{v(u)}/\F}\left(f(cu)\right)
=
\varkappa_{\E_{v(u)}/\F}
\left(
N_{\L_f/\F}
\left(1+u\sigma_f\right)
\right).
\end{equation}

We claim that for every $x \in \mathcal{O}_{\L_f}^{\times}$,
\begin{equation}\label{EquationAppendix9}
\varkappa_{\E_{v(u)}/\F}
\left(
N_{\L_f/\F}(x)
\right)
=
\varkappa_{\E_f/\L_f}(x).
\end{equation}
By (3) of Lemma \ref{LemmaAppendix1}, we have that for $x\in\mathcal{O}_{\L_f}^{\times}$,
\begin{align*}
\varkappa_{\E_{v(u)}/\F}
\left(
N_{\L_f/\F}(x)
\right)
&=
\varepsilon_{\F}
\left(
N_{\L_f/\F}(x)
\right)^{n-1}\\
&=
\varepsilon_{\L_f}(x)^{n-1}\\
&=
\varkappa_{\E_f/\L_f}(x),
\end{align*}
where the second equality follows from
\eqref{EquationAppendix0}.~This proves our claim.

Taking $x=1+u\sigma_f$ in
\eqref{EquationAppendix9} and applying
\eqref{EquationAppendix8}, we conclude that
\[
\left(-1\right)^{D(v(u))-n}
\varkappa_{\E_{v(u)}/\F}\left(f(cu)\right)
=
\varkappa_{\E_f/\L_f}
\left(1+u\sigma_f\right). \qedhere
\]
\end{proof}

The next main identity we obtain is the following. 
\begin{proposition}\label{PropositionAppendix3}
The quasi-character $\xi_{(f, \, \chi, \, \zeta)}$ satisfies
\[
\xi_{(f, \, \chi, \, \zeta)}(\sigma_f)
=
\varkappa_{\E_f/\L_f}(\sigma_f) \, 
\chi(\sigma_f).
\]
\end{proposition}
\begin{proof}
Recall that $\sigma_f=\beta_f^{n} \, \varpi_{\F}$.~By Corollary~\ref{UniformizerCorollary},
\[
\xi_{(f, \, \chi, \, \zeta)}(\beta_f)
=
\zeta \cdot \lambda_{\E_f/\F}(\psi_{\F}).
\]
Moreover, \eqref{CentralCharacterCondition} and the fact that
$\varkappa_{\E_f/\F}$ has order dividing two give
\[
\xi_{(f, \, \chi, \, \zeta)}
 \, \big|_{ \, \F^{\times}}
=
\omega_{(f, \, \chi, \, \zeta)}
\otimes
\varkappa_{\E_f/\F}.
\]
It follows that
\begin{align}
\xi_{(f, \, \chi, \, \zeta)}(\sigma_f)
&=
\xi_{(f, \, \chi, \, \zeta)}(\beta_f)^{n} \, \xi_{(f, \, \chi, \, \zeta)}(\varpi_{\F})
\notag\\
&=
\lambda_{\E_f/\F}(\psi_{\F})^{n} \, \zeta^{n} \, \omega_{(f, \, \chi, \, \zeta)}(\varpi_{\F}) \, \varkappa_{\E_f/\F}(\varpi_{\F})
\notag\\
&=
\lambda_{\E_f/\F}(\psi_{\F})^{n} \, \varkappa_{\E_f/\F}(\varpi_{\F}) \, \chi(\sigma_f).
\label{EquationAppendix10}
\end{align}

Set 
\[
\varpi_{\L_f} := \frac{\varpi_{\F}}{\sigma_f},
\]
which is a uniformizer of $\L_f$.~Using \eqref{LanglandsConstantTower} and (2) of Lemma \ref{LemmaAppendix1}, we have
\begin{align*}
&\lambda_{\E_f/\F}(\psi_{\F})^{n} \, \varkappa_{\E_f/\F}(\varpi_{\F}) \\
&\qquad=
\lambda_{\L_f/\F}(\psi_{\F})^{n^{2}} \, 
\lambda_{\E_f/\L_f}
\left(\psi_{\L_f}\right)^{n} \,
\varkappa_{\L_f/\F}(\varpi_{\F})^{n} \,
\varkappa_{\E_f/\L_f}(\varpi_{\F}).
\end{align*}
Since $\L_f/\F$ is unramified quadratic,
\eqref{LanglandsConstantUnramified} gives
\[
\lambda_{\L_f/\F}(\psi_{\F})=-1.
\]
Furthermore, \cite[Proposition 10.1.5]{BF} tells us that $\varkappa_{\L_f/\F}$ is non-trivial and unramified.~This implies
\[
\varkappa_{\L_f/\F}(\varpi_{\F})=-1.
\]

From \cite{BF, Moy}, we have that
\[
\lambda_{\E_f/\L_f}
\left(\psi_{\L_f}\right)^{n}
=
\varkappa_{\E_f/\L_f}
\left(\varpi_{\L_f}\right).
\]
Therefore,
\[
\lambda_{\E_f/\F}(\psi_{\F})^{n}
\varkappa_{\E_f/\F}(\varpi_{\F}) = \left(-1\right)^{n^{2}+n}
\varkappa_{\E_f/\L_f}
\left(\varpi_{\L_f}\right)
\varkappa_{\E_f/\L_f}(\varpi_{\F}) = \varkappa_{\E_f/\L_f}(\sigma_f).
\]
Substituting this identity into \eqref{EquationAppendix10} completes the proof of Proposition \ref{PropositionAppendix3}.
\end{proof}

\end{document}